\documentclass[12pt]{article}
\usepackage[small,compact]{titlesec}
\usepackage{amsmath,amssymb,amsfonts,mathabx,setspace}
\usepackage{graphicx,caption,epsfig,epsfig,wrapfig}
\usepackage{caption}
\usepackage[caption=false]{subfig}
\usepackage{float}
\usepackage{url,color,verbatim,algorithmicx,xcolor}
\usepackage[noend]{algpseudocode}
\usepackage[ruled,boxed]{algorithm} 

\usepackage{enumitem}
\usepackage{booktabs}  
\usepackage[sort,compress]{cite}
\usepackage[scaled]{helvet}
\usepackage[T1]{fontenc}
\usepackage[bookmarks=true, bookmarksnumbered=true, colorlinks=true,   pdfstartview=FitV,
linkcolor=blue, citecolor=blue, urlcolor=blue]{hyperref}
\usepackage{cleveref}
\usepackage[english]{babel}  
\usepackage{multirow}

\usepackage{tikz-cd}  
\usepackage{arydshln}
\usepackage[sort,compress]{cite}
\usepackage[normalem]{ulem}
\usepackage{mathtools}
\newcommand{\mathsout}[1]
{\bgroup\mathchoice
  {\sbox0{$\displaystyle{#1}$}%
    \usebox0\hspace{-\wd0}%
    \rule[0.5\ht0-0.5\dp0-.5pt]{\wd0}{1pt}}%
  {\sbox0{$\textstyle{#1}$}%
    \usebox0\hspace{-\wd0}%
    \rule[0.5\ht0-0.5\dp0-.5pt]{\wd0}{1pt}}%
  {\sbox0{$\scriptstyle{#1}$}%
    \usebox0\hspace{-\wd0}%
    \rule[0.5\ht0-0.5\dp0-.5pt]{\wd0}{1pt}}%
  {\sbox0{$\scriptscriptstyle{#1}$}%
    \usebox0\hspace{-\wd0}%
    \rule[0.5\ht0-0.5\dp0-.5pt]{\wd0}{1pt}}%
\egroup}
\def\bf{\mathbf{f}}

\def\bA{\mathbf{A}}  \def\ba{\mathbf{a}} 
  \def\bc{\mathbf{c}}
\def\bb{\mathbf{b}}

\def\bG{\mathbf{G}}
\def\calH{\mathcal{H} }
\def\bg{\mathbf{g}}

\def\brho{\boldsymbol{\rho}}
\def\bxi{\boldsymbol{\xi}}
\def\bSigma{\boldsymbol{\Sigma}}
\newcommand{\bphi}{\boldsymbol{\phi}}

\def\R{\mathbb{R}}                            
\def\P{\mathbb{P}}

\def\E{\mathbb{E}}

\newcommand{\norm}[1]{\left\|#1\right\|}
\newcommand{\innerp}[1]{\langle{#1}\rangle}

\newcommand{\mathspan}[1]{ \mathrm{span}\left\{ {#1} \right\} }

\newcommand{\argmin}[1]{\underset{#1}{\operatorname{arg}\operatorname{min}}\;}
\newcommand{\argmax}[1]{\underset{#1}{\operatorname{arg}\operatorname{max}}\;}

\newcommand\calR{\mathcal{R}}
\newcommand\calN{\mathcal{N}}

\def\calE{\mathcal{E}}
\def\Gbar{ {\overline{G}} }

\def\LGbar{ {\mathcal{L}_{\overline{G}}}  }

\def\calS{\mathcal{S}}
\def\calN{\mathcal{N}}

\def\calX{\mathcal{X}}
\def\spaceX{\mathbb{X}}
\def\spaceY{\mathbb{Y}}

\newcommand{\beqa}{\begin{eqnarray}}
\newcommand{\eeqa}{\end{eqnarray}}
\newcommand{\beqas}{\begin{eqnarray*}}
\newcommand{\eeqas}{\end{eqnarray*}}
\newcommand{\beq}{\begin{equation}}
\newcommand{\eeq}{\end{equation}}
\newcommand{\beqs}{\begin{equation*}}
\newcommand{\eeqs}{\end{equation*}}
\newcommand{\one}{{\rm{\mathbf{1}}}}

\newtheorem{theorem}{Theorem}

\newtheorem{assumption}[theorem]{Assumption}

\newtheorem{definition}[theorem]{Definition}
\newtheorem{example}[theorem]{Example}
\newtheorem{lemma}[theorem]{Lemma}

\newtheorem{remark}[theorem]{Remark}
\newenvironment{proof}[1][Proof]{\noindent\textbf{#1.} }{\ \rule{0.5em}{0.5em}}

\numberwithin{equation}{section}
\numberwithin{theorem}{section}

\usepackage{xcolor}
\definecolor{darkmagenta}{rgb}{0.55, 0.0, 0.55}
\colorlet{colorYO}{darkmagenta}

\title{Automatic reproducing kernel and regularization for learning convolution kernels}
\author{
Haibo Li 
\thanks{School of Mathematics and Statistics, The University of Melbourne, Parkville, VIC 3010, Australia. 
\href{mailto:haibo.li@unimelb.edu.au}{haibo.li@unimelb.edu.au}
} \ and \
Fei Lu
\thanks{Department of Mathematics, Johns Hopkins University, Baltimore, USA. 
\href{mailto:feilu@math.jhu.edu}{feilu@math.jhu.edu}
}
}

\begin{document}

\date{}
\maketitle
\vspace{-2mm}

\begin{abstract} 
 Learning convolution kernels in operators from data arises in numerous applications and represents an ill-posed inverse problem of broad interest. With scant prior information, kernel methods offer a natural nonparametric approach with regularization. However, a major challenge is to select a proper reproducing kernel, especially as operators and data vary. We show that the input data and convolution operator themselves induce an automatic, data-adaptive RKHS (DA-RKHS), obviating manual kernel selection. In particular, when the observation data is discrete and finite, there is a finite set of automatic basis functions sufficient to represent the estimators in the DA-RKHS, including the minimal-norm least-squares, Tikhonov, and conjugate-gradient estimators. We develop both Tikhonov and scalable iterative and hybrid algorithms using the automatic basis functions. Numerical experiments on integral, nonlocal, and aggregation operators confirm that our automatic RKHS regularization consistently outperforms standard ridge regression and Gaussian process methods with preselected kernels.
\end{abstract}

\tableofcontents

\section{Introduction}
Kernel functions play a fundamental role in defining operators between function spaces, enabling the representation of nonlocal or long-range interactions between variables. Such kernel-based operators permeate diverse fields: they describe nonlocal diffusion and peridynamic mechanics in partial differential equations (PDEs) \cite{benson2000application,silling2007peridynamic,du2012analysis,you2022data,d2020numerical,katiyar2020general,carrillo2019aggregation,chen2018heat}, govern anomalous transport in fractional diffusion and Lévy processes \cite{applebaum09,duan2015introduction}, and underpin advanced image-processing techniques \cite{buades2010image,gilboa2009nonlocal,lou2010image}. More recently, they have become central in operator-learning frameworks for scientific machine learning, from DeepONets \cite{lu2019deeponet} and Fourier neural operators \cite{li2020fourier,kovachki2023neural}, nonlocal neural networks \cite{wang2018non,anandkumar2020neural}, and kernel methods \cite{chen2021solving,owhadi2019kernel,darcy2025learning}. 

Motivated by these applications, a natural and challenging inverse problem arises: given pairs of inputs and outputs, how can one accurately recover the underlying kernel? We address this question in the linear setting, where the operator acts by convolution against a functional of the input function. By framing kernel recovery as a deconvolution inverse problem, we lay the groundwork for rigorous analysis and practical algorithms that learn these kernels directly from data, bridging the gap between classical inverse problems and modern data-driven operator learning.

\subsection{Problem statement}
We study the problem of estimating a convolution kernel $\phi: \mathcal{S} = [0,1]\to \R$ in the operator $R_\phi: \spaceX \to \spaceY$ of the form
\begin{equation}\label{eq:R_phi_g}
R_\phi[u](x) = \int_{\mathcal{S}} \phi(s)\,g[u](x,s)\,ds, \quad x \in \mathcal{X} = \{x_j\}_{j=1}^J \subset [0,1],
\end{equation}
based on discrete and noisy input-output pairs
\begin{equation}\label{eq:data}
\mathcal{D} = \{(u_k(y_i), f_k(x_j)) : \;1 \le k \le n_0,\; 1 \le i \le 3J, \; 1 \le j \le J\}.
\end{equation}
Here, $\{y_i\}_{i=1}^{3J}$ and $\{x_j\}_{j=1}^J$ are uniform meshes of $[-1,2]$ and $[0,1]$ with mesh sizes $\Delta x =y_{i+1}-y_i= x_{j+1}-x_j= \frac{1}{J}$, and these data are generated according to
\[
f_k(x_j) = R_\phi[u_k](x_j) + \epsilon_k(x_j), \quad
\epsilon_k(x_j) \stackrel{\text{i.i.d.}}{\sim} \mathcal{N}(0,\sigma^2/\Delta x).
\]
The input function space $\spaceX\subset L^2([-1,2])$ and the functional $g[u](x,s)$ are problem-specific (see Examples \ref{example:IntOpt}-- \ref{example:aggregation}). Note that $R_{\phi}$ can be a nonlinear functional of $u$, but it depends linearly on $\phi$. In this study, we consider a fixed discrete observation set $\calX$ and set the output function space to be $\spaceY = L^2_{\nu}(\calX)$ with an atomic measure $\nu$ defined by $\nu(\{x_j\})= 1/J$. When $\calX$ is a continuum set $[0,1]$ with Lebesgue measure, the corresponding output space is $L^2([0,1])$, the noise is white, and the minimax convergence rates in the sample size $n_0$ have been studied in \cite{zhang2025minimax}.

Such deconvolution-type problems arise in a wide range of applications, and we present three representative examples. 
\begin{example}[Integral operator]\label{example:IntOpt}
Estimate $\phi:\mathcal{S}\to\mathbb{R}$ in the integral operator
\[
R_\phi[u](x) = \int_{[-1,2]} \phi(x-y)\,u(y)\,dy
= \int_{[0,1]} \phi(s)\,u(x-s)\,ds
\]
with input space $\spaceX=C([-1,2])$.  In the form \eqref{eq:R_phi_g}, the functional is $g[u](x,s) = u(x-s)$ for $(x,s)\in\mathcal{X}\times\mathcal{S}$. 
The input $u$ is a sample of the stochastic process
$u(y) = \sum_{n=2}^{n_u} X_n\cos(2\pi n y) 
$ 
with $n_u\le+\infty$, where the coefficients $\{X_n\}$ are independent Gaussian random variables $N(0,4\sigma_n^2)$ with $\sum_n n \sigma_n<+\infty$.  
\end{example}

\begin{example}[Nonlocal operator]\label{example:nonlocal}
Estimate $\phi:\mathcal{S}\to\mathbb{R}$ in the nonlocal operator:
$$
R_\phi[u](x) = \int_{|x'|\le 1} \phi(|x'|)\bigl(u(x+x')-u(x)\bigr)\,\nu(dx')
= \int_{[0,1]} \phi(s)\,g[u](x,s) \,ds,
$$
with input space $\spaceX=C^1([-1,2])$ and 
$g[u](x,s) = u(x+s)+u(x-s)-2u(x)$. This operator arises in peridynamics {\rm\cite{LAY22,you2022data,you2024nonlocal}} and the Fokker-Planck equation of L\'evy processes {\rm\cite{applebaum09}}. 
\end{example}

\begin{example}[Aggregation operator]\label{example:aggregation}
 Consider the \emph{aggregation operator} $R_\phi[u] = \nabla\cdot(u\nabla\Phi*u)$  in the mean-field equation $\partial_t u = \nu \Delta u + \nabla\cdot(u\nabla\Phi*u)$ for interacting particle systems {\rm \cite{carrillo2019aggregation,LangLu21id}}. Let $\Phi$ be a radial potential supported on $\mathcal{S}$ and set $\phi=\Phi'$. For $u\in \spaceX = C^1([-1,2])$, one has
$$
R_\phi[u](x) = \int_{\mathbb{R}} \phi(|x'|)\frac{x'}{|x'|}\,\partial_x\bigl[u(x-x')u(x)\bigr] \,\, \nu(dx')
= \int_{\mathcal{S}} \phi(s)g[u](x,s)
\,ds 
$$
with $g[u](x,s) = \partial_x[ u(x-s)u(x)] - \partial_x[ u(x+s)u(x)]$. We consider  input functions $u$ to be random probability density functions 
$	u(x)= 1 + \sum_{n=1}^{n_u}
\sigma_n\,\zeta_n\,\cos\bigl(2\pi n\,x\bigr), 
$
where $\{\zeta_n\}_{n\ge1}$ are i.i.d.\ random signs {\rm (i.e., $\P(\zeta_n=\pm1)=\tfrac12$)}, and $\sigma_n>0$ with $\sum_n n \sigma_n<1$. 
\end{example}

\subsection{Main results: automatic reproducing kernel and regularization}
\paragraph{Challenges in learning kernels.}
Learning convolution kernels from discrete, noisy observations is a severely ill-posed inverse problem: even small data perturbations can induce large estimation errors, making regularization indispensable. Moreover, with scant prior knowledge of the true kernel, a nonparametric framework is necessary, rendering the choice of regularization norm both critical and nontrivial.

Kernel methods are particularly suitable for such inverse problems, as they can non-parametrically approximate the unknown functions with regularization using reproducing kernel Hilbert spaces (RKHS). 
Hence, they have been widely used in machine learning and inverse problems, dating back from solving the Fredholm equations in \cite{wahba1977practical,wahba1973convergence} and functional linear regression \cite{wahba1990spline,yuan_cai2010} to the recent studies on learning dynamical systems \cite{darcy2025learning,feng2024learning}, one-shot stochastic differential equations \cite{darcy2023oneshot}, linear responses estimations \cite{zhang2020estimating}, and solvers for nonlinear PDEs \cite{chen2021solving} and PDEs on manifolds \cite{harlim2020kernel}, to name just a few. In particular, the representer theorem reduces the problem with finite data to a finite-dimensional form, enabling efficient computation and feature extraction. 

However, a major obstacle in kernel methods is the choice of the reproducing kernel. Standard options, such as Gaussian or Mat\'ern kernels, come with hyperparameters (e.g., bandwidth or smoothness order) that must be carefully tuned. This process is not only computationally expensive but also fails to exploit the specific structure of the inverse problem at hand. In particular, when learning kernels in operators, the variational normal operator may be rank-deficient or possess zero eigenvalues, rendering conventional kernel selection and hyperparameter tuning virtually intractable.
  
\paragraph{Main results.} 
To overcome this obstacle, we propose an \emph{automatic reproducing kernel} that is defined directly in terms of the data and the forward operator. By incorporating the normal operator from the variational formulation, our kernel automatically adapts to the geometry and spectral properties of the inverse problem.  The resulting data-adaptive (DA) RKHS has a closure that is the space in which we can identify the true convolution kernel. Moreover, we use the representer theorem to derive a set of \emph{automatic basis functions} that are adaptive to the finite discrete observations and are sufficient to represent the estimators in the DA-RKHS, including the minimal-norm least-squares, Tikhonov, and conjugate-gradient estimators. These basis functions make mesh-free regression possible and reveal the finite-dimensional nature of the seemingly infinite-dimensional inverse problem of deconvolution.     

Building on this theory, we develop two families of regularization algorithms for efficient implementations of the automatic reproducing kernel: 
\begin{itemize}
\item \emph{Tikhonov methods} based on matrix decomposition for small to medium datasets, with regularization parameters chosen via the L-curve or generalized cross-validation criteria.
\item \emph{Iterative and hybrid regularization schemes} that rely solely on matrix-vector products, which are scalable for large datasets. 
\end{itemize}

\paragraph{Notations.} Throughout, we use roman letters (e.g., $f,u,G$) and Greek letters (e.g., $\phi,\xi,\lambda$) to denote functions or scalars, with their meanings clear from context. Boldface symbols (e.g., $\mathbf{G},\mathbf{c},\mathbf{x}$) denote vectors or matrices; we write $\mathbf{c}=(c_i)$ or $\mathbf{c}(i)$ for its $i$-th component. We reserve $0$ for the zero function and $\mathbf{0}$ for the zero vector, and denote by $\mathbf{I}_k$ the $k\times k$ identity matrix. For a closed linear subspace $\mathcal{H}$, $P_{\mathcal{H}}$ is the orthogonal projection. Given any linear operator or matrix, $\mathcal{N}(\cdot)$ and $\mathcal{R}(\cdot)$ are its null and range spaces, respectively. Finally, for a bounded linear operator $T$ between Hilbert spaces, $T^*$ denotes its adjoint.

The structure of the paper is as follows. In \Cref{sec2}, we introduce the automatic reproducing kernel and automatic basis functions, and derive regularized estimators based on Tikhonov and iterative regularization methods. In \Cref{sec:discreteFn} and \Cref{sec4}, we propose practical algorithms for computing the estimators, including the approximations from discrete data and Tikhonov and iterative regularization algorithms for small and large datasets, respectively. We use three typical examples to illustrate the accuracy and efficiency of our methods in \Cref{sec5}. The conclusion is in \Cref{sec6}.

\section{Automatic reproducing kernel and basis functions}\label{sec2}
We first provide a brief review of reproducing-kernel methods for a variational formulation of the inverse problems. Leveraging the variational framework, we then introduce the automatic reproducing kernel. Next, we construct a finite set of automatic basis functions for regression and show that, despite the loss function being minimized over an infinite-dimensional function space, the minimizer actually resides in the finite-dimensional space spanned by these basis functions. 
In the next section, we build on this continuum analysis to develop practical discrete approximations, paving the way for efficient numerical implementation.

We make the following regularity assumption on data and the operator $R_\phi[u]$ in terms of the bivariate function $g[u]$ throughout this study. 
 \begin{assumption}\label{assumption}
  The functions $\{g[u_k]\}_{k=1}^{n_0}\subset L^2(\calX\times \calS)$ is uniformly bounded, i.e., $C_g:= \max_{1\leq k\leq n_0}\sup_{x\in \calX, s\in \calS}|g[u](x,s)| <\infty$.  
\end{assumption}

\subsection{Kernel methods for ill-posed variational inverse problems}\label{sec2.1}
We estimate the convolution kernel by a variational approach that minimizes the loss function over a hypothesis space $\mathcal{H}$: 
\begin{align} \label{eq:lossFn_empirical0} 
\widehat{\phi} = &\argmin{\phi\in\mathcal{H}} \mathcal{E}_\mathcal{D}(\phi),\quad  
\mathcal{E}_\mathcal{D}(\phi):= \frac{1}{n_0 }\sum_{k=1,j=1}^{n_0,J}| R_\phi[u_k](x_j)-f_k(x_j) | ^2 \Delta x,
\end{align}
The integral defining $R_\phi[u_k](x_j)$ requires semi-continuum data $\{g[u_k](x_j,s), s\in \calS\}_{k,j=1}^{n_0,J}$ that is discrete in $x$ and continuous in $s$, which in turn presumes access to the continuum data $u_k$. In practice, however, we only observe discrete data $u_k$ as in \eqref{eq:data} yielding the values to discrete $\{g[u_k](x_j,s_l); l=1,\dots,n_s\}_{k,j=1}^{n_0,J}$. In Section \ref{sec:discreteFn}, we use these discrete data to empirically approximate the integrals in ${R_\phi[u_k]}(x_j)$. 

Two preliminary tasks in this variational approach are to select a hypothesis space $\calH$ along with a representation of the function $\phi$, and select a penalty term for regularization, which is crucial for the deconvolution-type problem. 

Kernel methods achieve both tasks by selecting a reproducing kernel, which provides a reproducing kernel Hilbert space (RKHS) as the hypothesis space and provides an RKHS norm as the penalty term. Specifically, let $K$ be a reproducing kernel (a positive definite function on $\calS\times \calS$) and denote its RKHS by $H_K$. Each function in the RKHS can be represented by $\phi(s) = \sum_{l=1}^{n_s} c_l K(s_l,s) $, whose RKHS norm is $\|\phi\|_{H_K} =\sqrt{ \sum_{ij} c_i c_j K(s_i,s_j)}$. Here, the sample points $\{s_l\}_{l=1}^{n_s}$ must be properly chosen to extract enough features. Then, the minimizer of the quadratic loss function follows from solving the coefficients $\mathbf{c}=(c_1,\ldots, c_{n_s})$ via least squares with a penalty term depending on $\|\phi\|_{H_K}^2$. 

However, the choice of the reproducing kernel $K$ is a major challenge. The widely used reproducing kernels, such as Gaussian kernels, come with hyperparameters that must be carefully tuned along with the regularization strength parameter. This process is not only computationally expensive but also fails to leverage the specific structure of the forward operator $R_\phi[u]$. In particular, when the quadratic loss function is not strictly convex, the conventional kernel selection and hyperparameter tuning are challenging. 

We address this challenge in the next section by introducing an automatic reproducing kernel that is adaptive to the data and the forward operator.

\subsection{Automatic reproducing kernel and RKHS}\label{sec.2}
We first introduce a weighted function space $L^2_\rho(\calS)$, where the measure $\rho$ defined below quantifies the exploration of data to the unknown function through the functions $\{g[u_k](x,\cdot)\}_{k}$, hence it is referred to as an \emph{exploration measure}. 
 \begin{definition}\label{def:rho}
Given data satisfying {\rm Assumption \ref{assumption}}, let $\rho$ be a measure on $\calS$ with a density function with respect to the Lebesgue measure:
 \beq
 \label{eq:exp_measure}
 \dot\rho(s):=\frac{1}{n_0Z} \sum_{k=1}^{n_0} \int_\calX | g[u_k](x,s)| \nu(dx), \, \forall s\in \calS, 
 \eeq
 where $Z= \frac{1}{n_0} \sum_{k=1}^{n_0} \int_\calS \int_\calX |  g[u_k](x,s) | \nu(dx) ds$ is the normalization constant. 
\end{definition} 

The exploration measure plays the role of the probability measure $\rho_X$ in nonparametric regression of $f(x)= \E[Y\,\big| X=x]\in L^2_{\rho_X}(\calS)$ from data $\{(x_i,y_i)\}$ that are samples of the joint distribution $(X,Y)$ (see e.g., \cite{CuckerSmale02,Gyorfi06a}). Here, we use the $L^1$ norm of $g[u](\cdot,s)$; alternatively, one can also use the $ L^2$ norm, as in \cite{zhang2025minimax}, to relax the constraint on $g[u]$. It is particularly useful when treating singular kernels in nonlocal operators, which may not be square integrable with respect to the Lebesgue measure, but square integrable in $L^2_\rho$. For example, $\phi(s)= s^{-\alpha}\notin L^2([0,\delta])$ for $\alpha \in (\frac{1}{2},\frac{3}{2})$, but $\phi\in L^2_\rho(\calS)$ when $u_k\in C^2[0,1]$ with uniformly bounded second-order derivatives since  $\dot \rho(s)= O(s^2)$ for small $s$ since $g[u_k](x,s) = u_k(x+s)+u_k(x-s)-2 u_k(x)= u_k''(x)s^2/2 + o(s^2)$.

Next, we introduce the automatic reproducing kernel.  
\begin{definition}[Automatic reproducing kernel]
\label{def:Gbar}
The automatic reproducing kernel for estimating $\phi$ in the operator $R_\phi$ in \eqref{eq:R_phi_g} from data $\{g[u_k](x,\cdot)\}_{k}$ is the function $\Gbar:\calS\times \calS \to \R$ defined by 
  \begin{equation}  \label{eq:G-rho}
  \begin{aligned}
  \Gbar(s,s')&:= 
  \frac{G(s,s')}{\dot\rho(s)\dot\rho(s')}  \one_{\{\dot\rho(s)\dot\rho(s')>0\}}, \quad   G(s,s')&:= 	\frac{1}{n_0}\sum_{k=1}^{n_0} \int_\calX g[u_k](x,s)g[u_k](x,s') \nu(dx), 
  \end{aligned}
  \end{equation}
where $\dot\rho$ is defined in \eqref{eq:exp_measure}. 
\end{definition}

The next lemma shows that the automatic reproducing kernel comes from the quadratic term of the loss function $\calE_\mathcal{D}(\phi)$ in \eqref{eq:lossFn_empirical0}. 
In particular, the closure (in $L^2_\rho$) of its RKHS is the space in which the variational problem has a unique minimizer. Its proof is postponed to \Cref{sec:appd_A}. For notation simplicity, we  write $L^2_{\rho}(\calS)$ and $L^2_{\rho\otimes \rho}(\calS\times\calS)$ as $L^2_\rho$ and $L^2_{\rho\otimes \rho}$, respectively.
\begin{lemma}\label{thm:FSOI} 
Under Assumption {\rm\ref{assumption}}, the following statements hold true:
 \begin{itemize}
 \item[(a)]   $ \Gbar(s,s') $ is in $L^2_{\rho\otimes \rho}$ and symmetric, and the  operator $\LGbar: L^2_\rho \to L^2_\rho$ defined by  
 \beq
 \label{eq:LG}
 \LGbar\phi(s):=\int_{\calS} \phi(s')\Gbar(s,s')\rho(ds') 
 \eeq
  is compact, self-adjoint, and positive. Hence, its eigenvalues $\{\lambda_i\}_{i\geq 1}$ are nonnegative and its orthonormal eigenfunctions $\{ \psi_i\}_{i}$ form a complete basis of $L^2_\rho$.  
\item[(b)] The loss function $\calE_\mathcal{D}(\phi)$ in \eqref{eq:lossFn_empirical0} can be written as  
	 \begin{align}\label{eq:lossFn-sq}
	 \mathcal{E}_\mathcal{D}(\phi)
	 & = \innerp{\LGbar\phi,\phi}_{L^2_\rho} - 2\innerp{\phi^D,\phi}_{L^2_\rho}+ const.,
	 \end{align}
	 where $\phi^D$ comes from the Riesz representation, $\innerp{\phi^D,\phi}_{L^2_\rho}  =\frac{1}{n_0}\sum_{k=1}^{n_0}\innerp{R_\phi[u_k],f_k}_{\spaceY} $ for any $\phi\in L^2_\rho$. In particular, when the data is noiseless, $\phi^D= \LGbar \phi_*$ and the loss function $\calE_{\mathcal{D}}$ has a unique minimizer $\widehat{\phi} = \LGbar^{-1}\phi^D =P_H(\phi_*)$ in $H:=\overline{\mathrm{span}\{\psi_i\}_{i:\lambda_i>0}} = \mathcal{N}(\LGbar)^\perp \subset L^2_\rho$. 
\item[(c)] The RKHS of $\Gbar$ is $H_\Gbar:=\LGbar^{\frac{1}{2}}(L^2_\rho)$ with inner product 
	 $\innerp{\phi, \phi}_{H_\Gbar}=\innerp{\LGbar^{-\frac{1}{2}}\phi,\LGbar^{-\frac{1}{2}} \phi }_{L^2_\rho}. 
	 $ 
    We have $H=\overline{ H_\Gbar}$ with closure in $L^2_\rho$ and $\innerp{\phi,\LGbar\psi}_{H_\Gbar}= \innerp{\phi,\psi}_{L^2_\rho}$ for any $\phi\in H_\Gbar, \psi\in L^2_\rho$.  
	 \end{itemize}
\end{lemma}

Therefore, solving $\nabla \mathcal{E}(\phi)=2(\LGbar\phi-\phi^D)=0$ yields the formal solution $\LGbar^{-1}\phi^D$. This inverse exists in $H$ when $\phi^D \in \LGbar(L_{\rho}^2)$ (and is undefined otherwise). Because $\LGbar$ is compact and may even be rank-deficient, the variational inverse problem is ill-posed, making regularization essential for a stable and accurate approximation of the true $\phi$. The RKHS $H_\Gbar$ provides a natural regularization space: its closure is $H$, and it automatically filters out any component of $\phi^D$ not in $\LGbar(L_{\rho}^2)$, which arises solely from noise or model error (see \cite[Theorem 2.7]{chada2024data} for a decomposition of $\phi^D$).

\subsection{Automatic basis functions and Tikhonov regularization}\label{sec2.3}
In the RKHS $H_\Gbar$, we seek a regularized solution by minimizing 
 \begin{equation}\label{regu1}
  \min_{\phi \in H_{\Gbar} } \calE_\lambda(\phi) : =  \mathcal{E}_\mathcal{D}(\phi)+\lambda \norm{\phi}_{H_\Gbar}^2, 
 \end{equation}
where $\mathcal{E}_\mathcal{D}$ is given in \eqref{eq:lossFn_empirical0}. In practice, one must choose a finite set of basis functions to represent elements of $H_\Gbar$. A common practice is to use the reproducing kernel to set a hypothesis space $\calH= \mathspan{\Gbar(s_j,\cdot)}_{j=1}^{n_s} \subset H_\Gbar \subset L^2_\rho$, where the $\{s_j\}_{j=1}^{n_s}\subset \calS$ are sample points. However, this can introduce bias and may fail to capture key features of the underlying inverse problem (see Remark \ref{rmk:basis-G}).  

To overcome these limitations, we construct a finite collection of \emph{automatic basis functions} tailored to the semi-continuum observations $\{g[u_k](x_j,\cdot)\}_{k,j=1}^{n_0,J}$. In particular, we show that even though the minimization of the loss function is taken over an infinite-dimensional space, the minimizer actually lies within the finite-dimensional span of these automatic bases.  This result extends the classical finite-dimensional representer theorem for smoothing spline in \cite[Theorem 1.3.1]{wahba1990spline} to our data-adaptive setting. 
 
\begin{theorem}[Finite-dimensional representer]\label{thm:Finite-rep}
Given functions $\{g[u_k](x_j,\cdot)\}_{k,j=1}^{n_0,J}$, let 
\begin{equation}\label{eq:auto-basis}
\xi_{kj}(s) = \LGbar [\frac{g[u_k](x_j,\cdot)}{\dot\rho(\cdot)}](s) =  \int_\calS \Gbar(s,s') g[u_k](x_j,s')	ds'
\end{equation}
for each $k,j$. Then, $\xi_{kj}\in H_{\Gbar}$ and  $\innerp{\xi_{kj},\phi}_{H_{\Gbar}}=R_{\phi}[u_k](x_j)$.  Let 
\begin{equation}\label{eq:Sigma-f}
\begin{aligned}
\bSigma&=( \innerp{\xi_{kj},\xi_{k'j'}}_{H_{\Gbar}}) \in \R^{n_0J\times n_0J}, \quad \mathbf{f} = (f_k(x_j))\in\R^{n_0J} . 
\end{aligned}
\end{equation}
We have finite-dimensional representations for the estimators in \eqref{eq:lossFn_empirical0} and \eqref{regu1} as follows.  
\begin{itemize}
\item[(a)]  The least squares estimator with minimal $H_{\Gbar}$-norm is 
    \begin{equation}\label{eq:lse-mini}
        \widetilde{\phi} =   \argmin{\psi\in \argmin{\phi\in H_{\Gbar}} \calE_{\mathcal{D}}(\phi)}\|\psi\|_{H_{\Gbar}}^2   = \sum_{kj} \widetilde c_{kj}\xi_{kj} , \ \text{ with } 
        (\widetilde{c}_{kj}) =:\widetilde{\mathbf{c}} = \bSigma^\dag \mathbf{f}, 
    \end{equation}
    where $\widetilde{\bc}$ is the minimal 2-norm solution of $\min_{\mathbf{c}}\{\frac{1}{n_0J}\|\bSigma\mathbf{c}-\mathbf{f}\|_{2}^{2}\}$. 
\item[(b)]  The estimator of Tikhonov regularization with $H_{\Gbar}$-norm and $\lambda>0$ is  
\begin{equation}\label{eq:tik-estimator}
	  \widehat \phi_\lambda = \argmin{\phi \in H_{\Gbar} } \calE_\lambda(\phi) = \sum_{kj} \widehat c_{kj}\xi_{kj}, \quad \widehat{\mathbf{c}}_\lambda = 	( \bSigma^2 + n_0J \lambda \bSigma)^{\dagger} \bSigma \mathbf{f}, 
\end{equation} 
where the coefficient $\widehat{\mathbf{c}}_\lambda = (\widehat c_{kj})\in \R^{n_0J} $ solves  $\min_{\mathbf{c}}\{\frac{1}{n_0J}\|\bSigma\mathbf{c}-\mathbf{f}\|_{2}^{2}+\lambda \mathbf{c}^\top \bSigma  \mathbf{c}\} $. 
\end{itemize}    
\end{theorem}

\begin{proof}
	First, since $\tilde g_{kj}:=\frac{g[u_k](x_j,s')}{\dot\rho(s')}\in L^2_\rho$, we have $\xi_{kj}= \LGbar \tilde  g_{kj} \in H_{\Gbar} $. 
	
	Next, note that for every $\phi\in H_{\Gbar}$, Lemma \ref{thm:FSOI}(c) implies that $\innerp{\LGbar\psi,\phi}_{H_{\Gbar}} = \innerp{\psi,\phi}_{L^2_\rho}$ for any $\psi\in L^2_\rho$. Hence, 
	\begin{equation*}
   \innerp{\xi_{kj}, \phi}_{H_{\Gbar}} = \innerp{\LGbar \tilde  g_{kj}, \phi}_{H_{\Gbar}} =  \innerp{ \tilde  g_{kj}, \phi}_{L^2_\rho} = R_\phi[u_k](x_j). 		
	\end{equation*}
In other words, $\xi_{kj}$ is a representer of the bounded linear functional $R_\phi[u_k](x_j)$ on $H_{\Gbar}$. 

Also, for any $\phi\in H_{\Gbar}$, we can write it as 
\[\phi= \xi + \sum_{kj} c_{kj}\xi_{kj}, \quad \xi \perp \mathrm{span}\{\xi_{kj}\}_{k,j=1}^{n_0,J}. 
\]
Then, we have $\|\phi\|_{H_{\Gbar}}^2 =\mathbf{c}^\top \bSigma  \mathbf{c} +\|\xi\|_{H_{\Gbar}}^2 $ and 
$\innerp{\xi_{kj}, \phi}_{H_{\Gbar}} = ( \bSigma \mathbf{c})_{kj}$.  
As a result, the loss functions $\mathcal{E}_\mathcal{D}(\phi)$ and $\mathcal{E}_\lambda(\phi)$ can be written as 
\begin{align*}
	 \mathcal{E}_\mathcal{D}(\phi) & = \frac{1}{n_0J}\sum_{kj} |f_k(x_j)- \innerp{\xi_{kj}, \phi}_{H_{\Gbar}}|^2 = \frac{1}{n_0J}\|\bSigma\mathbf{c}-\mathbf{f}\|_{2}^{2};  \\
	\mathcal{E}_\lambda(\phi)  & = \frac{1}{n_0J}\|\bSigma\mathbf{c}-\mathbf{f}\|_{2}^{2}+\lambda \mathbf{c}^\top \bSigma  \mathbf{c}  + \lambda\|\xi\|_{H_{\Gbar}}^2.
 \end{align*} 

Therefore, the minimizer $ \widetilde{\phi}$ of $\mathcal{E}_\mathcal{D}(\phi) = \frac{1}{n_0J}\|\bSigma\mathbf{c}-\mathbf{f}\|_{2}^{2}$ with minimal $H_\Gbar$-norm in \eqref{eq:lse-mini} has a coefficient $\widetilde{\bc}$ that can be solved with by the minimizer of $\frac{1}{n_0J}\|\bSigma\mathbf{c}-\mathbf{f}\|_{2}^{2}$ with minimal $\|\mathbf{c}\|_2$.  

Also, the minimizer of $\mathcal{E}_\lambda(\phi)$ solves $(\frac{1}{n_0J} \bSigma^2  + \lambda \bSigma) \mathbf{c} =\bSigma  \mathbf{f} $, which gives \eqref{eq:tik-estimator}. 
\end{proof}

Note that the matrix $\bSigma$ in \eqref{eq:Sigma-f} can be either singular or non-singular. If it is non-singular, the Tikhonov regularized estimator in \eqref{eq:tik-estimator} becomes $\mathbf{c}_{ridge} = (\bSigma + n_0J \lambda I )^{-1}\mathbf{f}$ after canceling out $\bSigma$, which is the widely used ridge regularized estimator. However, when $\bSigma$ is singular, this Tikhonov regularized estimator is different from the ridge estimator. It is $\mathbf{c}_\lambda= (\bSigma^2+ n_0J \lambda \bSigma )^{\dag}\bSigma \mathbf{f} = (\bSigma+ n_0J \lambda I )^{-1}P_{\mathcal{N}(\bSigma)^\perp} \mathbf{f} $, which prevents the error in $P_{\calN(\bSigma)}\mathbf{f}$ from contaminating the estimator. In contrast, the ridge regularized estimator will be contaminated by the error $P_{\calN(\bSigma)}\mathbf{f}$ and would lead to disastrous results in the small noise limit \cite{LangLu23small,chada2024data}. Additionally, in either case, our Tikhonov solution in \eqref{eq:tik-estimator} converges to the least squares solution with the minimal norm as $\lambda\to 0$. 

Importantly, a singular $\bSigma$ does not imply multiple minimizers for the loss function over the function space, though it leads to multiple minimizers in the coefficient space. As the next remark shows, all the coefficient minimizers correspond to the same function minimizer because when $\bSigma$ is singular, the basis functions $\{\xi_{kj}\}$ are linearly dependent. In short, the loss function $\mathcal{E}_\mathcal{D}$ always has a unique minimizer in $H_{\Gbar}$ regardless of $\bSigma$ being singular or not.

\begin{remark}\label{unq_phi}
When $\bSigma$ is singular, there are infinitely many $\bc$ minimizing $\frac{1}{n_0J}\|\bSigma\mathbf{c}-\mathbf{f}\|_{2}^{2}$, but all such minimizers lead to the unique minimizer in $H_{\Gbar}$. Equivalently, the set $\{\phi=\sum_{kj}c_{kj}\xi_{kj}: \mathbf{c}=(c_{kj})\in\mathcal{C}\}$ contains only one element, where $\mathcal{C}=\widetilde{\mathbf{c}}+\mathcal{N}(\bSigma)$ is the set of all minimizers of $\|\bSigma\mathbf{c}-\mathbf{f}\|_{2}^{2}$. To see it, let $\phi_1$ and $\phi_2$ be such two elements with corresponding $\mathbf{c}_1, \mathbf{c}_2\in\mathcal{C}$. Then $\mathbf{c}_1-\mathbf{c}_2\in\mathcal{N}(\bSigma)$ and $\phi_1-\phi_2=\sum_{kj}(\mathbf{c}_1-\mathbf{c}_2)(kj)\xi_{kj}$. Therefore, $\|\phi_1-\phi_2\|_{H_{\Gbar}}^2 =(\mathbf{c}_1-\mathbf{c}_2)^\top \bSigma (\mathbf{c}_1-\mathbf{c}_2)=0 $, leading to $\phi_1=\phi_2$. The same conclusion holds for the regularized loss $\frac{1}{n_0J}\|\bSigma\mathbf{c}-\mathbf{f}\|_{2}^{2}+\lambda\mathbf{c}^{\top}\bSigma\mathbf{c}$. In other words, when $\bSigma$ is singular, the basis functions $\{\xi_{kj}\}$ are linearly dependent, so there are multiple coefficients that represent the same function minimizer. 
\end{remark}

\begin{remark}[Basis functions via the reproducing kernel.]\label{rmk:basis-G}
A default approach to solve \eqref{regu1} is to use basis functions of the reproducing kernel $\Gbar$ since  $H_{\Gbar}=\overline{\mathrm{span}\{\Gbar(s,\cdot)\}_{s\in \calS}}$. That is, take sample points $\{s_l\}_{l=1}^{n_s}\subset \calS$ and set hypothesis space to be $\calH= \mathspan{\Gbar(s_l,\cdot)}_{l=1}^{n_s}$. For any $ \phi(s)=\sum_{l=1}^{n_s} a_l \Gbar(s_l,s)\in \calH$, the square of its RKHS norm is  
 \begin{equation}\label{eq:normRKHS-a}
 \begin{aligned}
 \|\phi\|_{H_\Gbar}^2 &= \|\sum_{l=1}^{n_s} a_l \Gbar(s_l,\cdot)\|_{H_\Gbar}^2= 
 \sum_{l,l'=1}^{n_s} a_l a_{l'} \innerp{\Gbar(s_l,\cdot),\Gbar(s_l,\cdot)}_{H_\Gbar} = \mathbf{a}^\top \overline{\bG} \ba, \\
 \overline{\bG} (l,l') &= \innerp{\Gbar(s_l,\cdot),\Gbar(s_l,\cdot)}_{H_\Gbar} = \Gbar(s_l,s_{l'}), \quad  1\leq l,l'\leq {n_s}.   	
 \end{aligned}
 \end{equation}
 Then, the minimizer of $ \mathcal{E}_\mathcal{D}(\phi)+\lambda \norm{\phi}_{H_\Gbar}^2 = \ba^\top \bA \ba - 2\ba^\top \bb + Const +\lambda  \mathbf{a}^\top \overline{\bG} \ba$ is 
 \begin{equation}\label{eq:regu_loss_a}
 \begin{aligned}
\widehat \ba_\lambda & = (\bA+ \lambda \overline{\bG})^{\dagger}\bb,   \quad \text{ with }
 \\
 \bA(l,l') & = \innerp{\LGbar \Gbar(s_l,\cdot),\Gbar(s_{l'},\cdot)}_{L^2_\rho} = \int_\calS \int_\calS \Gbar(s_l,s)\Gbar(s_{l'},s') G(s,s') ds ds', \\ 
 \bb(l) &= \frac{1}{n_0} \sum_{k=1}^{n_0}\innerp{R_{\Gbar(s_l,\cdot)}[u_k],f}_{\spaceY} = \int_\calS \Gbar(s_l,s) \frac{1}{n_0} \sum_{k=1}^{n_0}  \int_\calX g[u_k](x,s) f_k(x) \nu(dx) ds. 
 \end{aligned}
 \end{equation}
 The major difficulty is selecting the sample points $\{s_j\}$ such that the resulting basis functions capture all the features in the data. In practice, this only succeeds when $\{s_j\}$ align exactly with the $x$-mesh, at which point the estimator coincides with the automatic basis estimator. By contrast, the automatic basis functions in Theorem {\rm\ref{thm:Finite-rep}} are guaranteed to extract all the features available in the data. Therefore, throughout this study, we employ the automatic basis functions.  
 \end{remark}

\subsection{Conjugate gradient and iterative regularization}\label{sec2.4}
Iterative regularization methods circumvent the computationally expensive matrix inversions or decompositions for Tikhonov regularization in \eqref{eq:tik-estimator} by minimizing the loss function on a sequence of growing subspaces with early stopping; see, e.g., \cite[Ch.~3.3]{engl1996regularization}. 
These subspaces are designed to progressively capture the dominant features of the true solution, and early stopping prevents the inclusion of noise-dominated directions.

To design iterative regularization methods adapted to the automatic basis functions, we can apply the conjugate gradient (CG) method (see, e.g., \cite[Ch. 7]{engl1996regularization}) to the normal equation of the least squares problem
\begin{equation}\label{LSE0}
  \argmin{\phi \in H_{\Gbar}} n_0J\calE_\mathcal{D}(\phi)= \|T\phi-\mathbf{f}\|_{2}^2, 
\end{equation}
where $T$ is the linear operator
\begin{equation}
  T: H_{\Gbar}\rightarrow (\mathbb{R}^{n_0J}, \langle \cdot, \cdot \rangle_{2}), \quad
  \phi \mapsto (\langle \xi_{kj},\phi\rangle_{H_{\Gbar}}) .
\end{equation}
At each iteration, the CG method selects a new search direction that is conjugate with respect to the normal operator $T^*T$ to all previous directions. In particular, CG is essentially a Krylov subspace method, where the solution to \eqref{LSE0} in the $l$-th CG iteration with initial guess $\phi_0=0$ is 
  \begin{equation}\label{eq:cg_solu_fn}
       \phi_{l}=\argmin{\phi\in \mathcal{H}_l}\|T\phi-\mathbf{f}\|_{2},  \quad \mathcal{H}_{l}:=\mathrm{span}\{(T^{*}T)^{i}T^{*}\mathbf{f}\}_{i=0}^{l-1}.
    \end{equation}
Each $\mathcal{H}_l$ is a subspace of $H_\Gbar$ and we call it the $l$-th \emph{RKHS-Krylov subspace}.

The following theorem shows the implementation of the above CG iteration in the coefficient space of the automatic basis functions. 
\begin{theorem}[Conjugate gradient solutions]\label{prop:representorIR}
At the $l$-th iteration, the CG solution in \eqref{eq:cg_solu_fn} is $\phi_{l}=\pi(\bc)=\sum_{kj}\bc_l(kj)\xi_{kj}$ with 
\begin{equation}\label{LS_sub}
  \bc_l = \argmin{\bc \in \mathcal{K}_l} \|\bSigma \bc - \bf \|_2, \quad
  \mathcal{K}_l:= \mathrm{span}\{\bSigma^{i}\bSigma^{\dag}\mathbf{f}\}_{i=1}^{l},
\end{equation}
and $\calH_l=\pi(\mathcal{K}_l)= \mathrm{span}\{\sum_{kj}\mathbf{a}_i(kj)\xi_{kj}\}_{i=1}^l$, where  $ \mathbf{a}_{i}=\bSigma^{i}\bSigma^{\dag}\mathbf{f}$ and $\pi$ is the linear operator  
 \begin{equation}\label{embed_pi}
    \pi: \mathbb{R}^{n_0J}\rightarrow H_0: =\mathrm{span}\{\xi_{kj}\}_{k,j=1}^{n_0,J} \subset H_{\Gbar}, \quad \mathbf{c}\mapsto \sum_{kj}c_{kj}\xi_{kj}.
  \end{equation}
In particular, $T\phi=\bSigma\bc$ for any $\phi=\sum_{kj} c_{kj}\xi_{kj}+\xi$ with $\xi\in H_{0}^{\perp}$ and $\mathbf{c}=(c_{kj})\in\mathbb{R}^{n_0J}$, $T\circ \pi = \bSigma$ and $T^*=\pi\circ(\bSigma^\dag \bSigma)$, implying the following two commutative diagrams:
  \begin{equation}\label{diags}
    \begin{tikzcd}
      H_{\Gbar} \arrow[r, "T"] \arrow[from=d, "\pi"]
      & \mathbb{R}^{n_0J}  \\
      \mathbb{R}^{n_0J} \arrow[ur, "\bSigma"']
    \end{tikzcd}
    \quad  \quad \quad \quad
    \begin{tikzcd}
      H_{\Gbar}  \arrow[from=d, "\pi"]
      & \mathbb{R}^{n_0J} \arrow[l, "T^{*}"'] \arrow[dl, "\bSigma^{\dag}\bSigma"] \\
      \mathbb{R}^{n_0J} 
    \end{tikzcd} .
  \end{equation}
\end{theorem}

\begin{proof}
  The proof includes the following four steps.

   Step 1: Prove that $T\phi=\bSigma\bc$ for any $\phi=\sum_{kj} c_{kj}\xi_{kj}+\xi$ with $\xi\in H_{0}^{\perp}$ and $\mathbf{c}=(c_{kj})\in\mathbb{R}^{n_0J}$, which implies  $T\circ \pi = \bSigma$. By the definition of $T$, it follows that $T\xi_{kj}=(\langle \xi_{k'j'},\xi_{kj}\rangle_{H_{\Gbar}})$ with $1\leq k'\leq n_0$ and $1\leq j'\leq J$. Thus, recalling that $\bSigma=( \innerp{\xi_{kj},\xi_{k'j'}}_{H_{\Gbar}})$,  we have $T\phi=T(\sum_{kj} c_{kj}\xi_{kj})=\sum_{kj} c_{kj}T\xi_{kj}= \bSigma\mathbf{c} $. 
 
  Step 2: Show that $T^* = \pi \circ \bSigma^{\dag}\bSigma $. 
  It suffices to show that for any $\mathbf{y}\in\mathbb{R}^{n_0J}$, in the decomposition $T^{*}\mathbf{y}=\sum_{kj}a_{kj}\xi_{kj}+\bar{\xi}$ with $\bar{\xi}\in H_{0}^{\perp}$, we have $\bar{\xi}=0$ and $(a_{kj})=:\mathbf{a}=\bSigma^{\dag}\bSigma\mathbf{y}$.  By the adjoint identity $\langle T\phi, \mathbf{y} \rangle_{2}=\langle \phi, T^{*}\mathbf{y} \rangle_{H_{\Gbar}}$ for any $\phi=\sum_{kj} c_{kj}\xi_{kj}+\xi$, we have 
  \begin{align*}
    \langle \bSigma\bc, \mathbf{y} \rangle_{2}=\langle \sum_{kj} c_{kj}\xi_{kj}+\xi, \sum_{kj}a_{kj}\xi_{kj}+\bar{\xi} \rangle_{H_{\Gbar}} \ \ \Leftrightarrow \ \ 
    \mathbf{c}^{\top}\bSigma\mathbf{y} = \mathbf{c}^{\top}\bSigma\bar{\mathbf{a}} + \langle \xi, \bar{\xi}\rangle_{H_{\Gbar}} 
  \end{align*}
  for all $\bc\in \R^{n_0J}$ and $\xi\in H_0^\perp$. Thus, taking $\bc =\mathbf{0}$, we have $\langle \xi, \bar{\xi}\rangle_{H_{\Gbar}}=0$ for all $\xi\in H_0^\perp$. Combining with $\bar{\xi}\in H_{0}^\perp$, we get $\bar{\xi}=\mathbf{0}$. Then,  we have $\mathbf{c}^{\top}\bSigma\mathbf{y} = \mathbf{c}^{\top}\bSigma \mathbf{a}$ for all  $\bc\in \R^{n_0J}$, resulting in $\bSigma\mathbf{y}=\bSigma\mathbf{a}$, or equivalently, $\mathbf{a}\in\bSigma^{\dag}\bSigma\mathbf{y}+\mathcal{N}(\bSigma)$. But any two choices of $\mathbf{a}$ differing by an element of $\mathcal{N}(\bSigma)$ give the same $\phi\in H_\Gbar$ (see Remark \ref{unq_phi}). Therefore, we can take $\mathbf{a}=\bSigma^{\dag}\bSigma\mathbf{y}$.

  Step 3: Compute $(T^{*}T)^{i}T^{*}\mathbf{f}$.  Note that $T^{*}\circ\bSigma=(\pi\circ \bSigma^{\dag}\bSigma)\bSigma=\pi\circ \bSigma$ since $\bSigma^{\dag}\bSigma= \bSigma\bSigma^{\dag}$. Using $T\circ\pi=\bSigma$, we have 
  \begin{equation*}
    (T^{*}T)^{i}\circ \pi = (T^{*}T)^{i-1}\circ T^{*} \circ \bSigma
    = (T^{*}T)^{i-1}\circ \pi \circ \bSigma =\cdots = \pi \circ \bSigma^{i}. 
  \end{equation*}
  Now we have $T^{*}\mathbf{f}=\pi(\bSigma^{\dag}\bSigma\mathbf{f})=\pi(\bSigma\bSigma^{\dag}\mathbf{f})$, and 
  \begin{align*}
    (T^{*}T)^{i}T^{*}\mathbf{f} = (T^{*}T)^{i}\circ \pi(\bSigma^{\dag}\bSigma\mathbf{f})
    = \pi(\bSigma^{i}\bSigma^{\dag}\bSigma\mathbf{f}) 
    = \pi(\bSigma^{i+1}\bSigma^{\dag}\mathbf{f}).
  \end{align*}
  Therefore, the $l$-th RKHS-Krylov subspace is  
  \begin{equation*}
  \calH_{l}=  \mathrm{span}\{(T^{*}T)^{i}T^{*}\mathbf{f}\}_{i=0}^{l-1}=\pi(\mathrm{span}\{\bSigma^{i}\bSigma^{\dag}\mathbf{f}\}_{i=1}^{l}) = \pi(\mathcal{K}_l).
  \end{equation*}

  Step 4: Prove $\phi_{l}=\pi(\mathbf{c}_l)$ with $\bc_l$ in \eqref{LS_sub}, i.e., $\bc_l= \argmin{\bc \in \mathcal{K}_l} \|\bSigma \bc - \bf \|_2$. From the above results, we have
  \begin{equation*}
    \phi_{l}=\argmin{\phi\in\pi(\mathcal{K}_l)}\|T\phi-\mathbf{f}\|_{2}
    = \argmin{\phi=\pi(\mathbf{c}) \atop \mathbf{c}\in\mathcal{K}_l}\|T\circ\pi(\mathbf{c})-\mathbf{f}\|_{2}
    = \argmin{\phi=\pi(\mathbf{c}) \atop \mathbf{c}\in\mathcal{K}_l}\|\bSigma\mathbf{c}-\mathbf{f}\|_2 .
  \end{equation*}
  It follows that $\phi_{l}=\pi(\mathbf{c}_l)$.
\end{proof} 

\Cref{prop:representorIR} implies that $\calR(T)\subset\calR(\bSigma)$. Furthermore, noting that $\calN(T)=H_0^\perp$, we have $\mathrm{dim}(\calR(T))=\mathrm{dim}(H_{\Gbar}/\calN(T))=\mathrm{dim}(H_0)=\mathrm{rank}(\bSigma)$, leading to $\calR(T)=\calR(\bSigma)$.
Then, we have $P_{\calN(\bSigma)^{\perp}}\mathbf{f}\in\calR(T)$ and $P_{\calN(\bSigma)}\mathbf{f}\perp\calR(T)$, leading to $\argmin{\phi \in H_{\Gbar}}\|T\phi-\mathbf{f}\|_{2}=\argmin{\phi \in H_{\Gbar}}\|T\phi-P_{\calN(\bSigma)^{\perp}}\mathbf{f}\|_{2}$ and 
 \begin{equation}\label{ls_krylov}
  \phi_{l} = \argmin{\phi\in\mathcal{H}_l}\|T\phi-P_{\calN(\bSigma)^{\perp}}\mathbf{f}\|_{2} =
  \argmin{\phi=\pi(\mathbf{c}) \atop \mathbf{c}\in\mathcal{K}_l}\|\bSigma\mathbf{c}-P_{\calN(\bSigma)^{\perp}}\mathbf{f}\|_2. 
\end{equation}
Therefore, we can obtain $\phi_{l} = \pi(\bc_l)$ by computing $\bc_l = \argmin{\bc \in\mathcal{K}_l}\|\bSigma\bc-P_{\calN(\bSigma)^{\perp}}\mathbf{f}\|_{2}$.

In practice, rather than applying CG directly, we use the Golub-Kahan bidiagonalization (GKB) to explicitly construct the solution subspace $\mathcal{K}_l$ and solve \eqref{ls_krylov} iteratively. This approach is mathematically equivalent to CG but avoids explicitly forming $T^*T$, which is more numerically stable, and the convergence of iterates can be further stabilized using the hybrid regularization method; see \cite{Kilmer2001,caruso2019convergence} for more details. 

\paragraph{Derivation of the GKB method.}
The recursive relations of GKB for $\{T,P_{\calN(\bSigma)^{\perp}}\mathbf{f}\}$ is given by:
\begin{equation}\label{GKB0}
  \begin{cases}
    \beta_{1}\mathbf{p}_1 = P_{\calN(\bSigma)^{\perp}}\mathbf{f}, \\
    \alpha_i\psi_i = T^{*}(\mathbf{p}_{i})-\beta_{i}\psi_{i-1}, \\
    \beta_{i+1}\mathbf{p}_{i+1}= T(\psi_i)-\alpha_i\mathbf{p}_{i},
  \end{cases}
\end{equation}
where $\psi_0:=0$, and $\{\alpha_i,\beta_i\}$ are computed such that $\{\psi_i\}\subset H_{\Gbar}$ and $\{\mathbf{p}_{i}\}\subset\mathbb{R}^{n_0J}$ are orthonormal, and $\mathrm{span}\{\psi_{i}\}_{i=1}^{l}=\mathcal{H}_{l}$. Let $\tilde{\pi}=\pi|_{\calN(\bSigma)^{\perp}}$, where $\pi$ is defined in \eqref{embed_pi}. Note that $\tilde{\pi}$ is injective, and the two commutative diagrams in \eqref{diags} still hold by replacing $\mathbb{R}^{n_0J}$ with $\calN(\bSigma)^{\perp}$. By \Cref{prop:representorIR}, for any $\psi_i$, there exist a unique $\mathbf{q}_{i}\in\calN(\bSigma)^{\perp}$ such that $\tilde{\pi}(\mathbf{q}_{i})=\psi_i$, and $\langle \psi_i, \psi_j \rangle_{H_{\Gbar}}=\mathbf{q}_{i}^{\top}\bSigma \mathbf{q}_{j}$. Using $T^{*}u_i=\tilde{\pi}(\bSigma^{\dag}\bSigma u_i)$, we have
\begin{equation*}
  \begin{cases}
    \alpha_i\tilde{\pi}(\mathbf{q}_{i}) = \tilde{\pi}(\bSigma^{\dag}\bSigma \mathbf{p}_{i})-\beta_{i}\tilde{\pi}(q_{i-1}), \\
    \beta_{i+1}\mathbf{p}_{i+1}= T\circ \tilde{\pi}(\psi_i)-\alpha_i\mathbf{p}_{i}.
  \end{cases}
\end{equation*}
Using $T\circ\tilde{\pi}=\bSigma$, we obtain the practical GKB recursive relations:
\begin{equation} \label{GKB1}
  \begin{cases}
    \beta_{1}\mathbf{p}_1 = P_{\calN(\bSigma)^{\perp}}\mathbf{f}, \\
    \alpha_i \mathbf{q}_{i} = \mathbf{p}_{i}-\beta_{i}\mathbf{q}_{i-1}, \\
    \beta_{i+1}\mathbf{p}_{i+1}= \bSigma \mathbf{q}_{i} -\alpha_i \mathbf{p}_{i},
  \end{cases}
\end{equation}
where $q_0:=\mathbf{0}$, and $\{\alpha_i,\beta_i\}$ are computed such that $\{\mathbf{q}_{i}\}\subset \calN(\bSigma)^{\perp}$ and $\{\mathbf{p}_{i}\}$ are $\bSigma$-orthonormal and 2-orthonormal, respectively. Here, to get the second relation of \eqref{GKB1}, we have used $\mathbf{p}_{i}\in\calN(\bSigma)^{\perp}$, which can be verified by mathematical induction.

Note that $(\calN(\bSigma)^{\perp}, \langle\cdot,\cdot\rangle_{\bSigma})$ is a finite-dimensional Hilbert space with inner product $\langle \mathbf{x},\mathbf{x}'\rangle_{\bSigma}:=\mathbf{x}^\top\bSigma \mathbf{x}$. The following theorem shows that GKB iteratively constructs a $\bSigma$-orthonormal basis of $\mathcal{K}_l$ in $(\calN(\bSigma)^{\perp}, \langle\cdot,\cdot\rangle_{\bSigma})$. The proof is in \Cref{sec:appd_A}. 

\begin{theorem}\label{thm:gkb}
  Following the notations in \Cref{prop:representorIR}, define the linear operator:
  \begin{equation}
    \widetilde{T}: (\calN(\bSigma)^{\perp}, \langle\cdot,\cdot\rangle_{\bSigma}) \rightarrow
    (\calN(\bSigma)^{\perp}, \langle\cdot,\cdot\rangle_{2}), \quad
    \mathbf{x} \mapsto \bSigma \mathbf{x}.
  \end{equation}
  Then \eqref{GKB1} is the recursive relations of the GKB for $\{\widetilde{T}, P_{\calN(\bSigma)^{\perp}}\mathbf{f}\}$, and the following properties hold:
  \begin{enumerate}
    \item[(a)] The two groups of vectors $\{\mathbf{q}_{i}\}_{i=1}^{l}$ and $\{\mathbf{p}_{i}\}_{i=1}^{l}$ are $\bSigma$-orthonormal and 2-orthonormal bases of the Krylov subspace
    \begin{equation}
      \mathcal{K}_{l}(\bSigma,P_{\calN(\bSigma)^{\perp}}\mathbf{f}):=\mathrm{span}\{\bSigma^{i}P_{\calN(\bSigma)^{\perp}}\mathbf{f}\}_{i=0}^{l-1}
      =\mathcal{K}_l.
    \end{equation}
    \item[(b)] Let the terminate step of GKB be $l_t=\argmin{i\geq 1}\{\alpha_{i+1}\beta_{i+1}=0\}$. Then $l_t\leq\mathrm{rank}(\bSigma)$, and $\phi_{l_t}=\widetilde{\phi}$, the LS estimator with minimal $H_{\Gbar}$-norm defined in \eqref{eq:lse-mini}.
    \item[(c)] Let the residual be $\mathbf{r}_l=T\phi_l-\mathbf{f}$. Then $\|\mathbf{r}_l\|_2=\|\bSigma\mathbf{c}_l-\mathbf{f}\|_2$ and $\|\phi_{l}\|_{H_{\Gbar}}=\|\mathbf{c}_l\|_{\bSigma}$, and $\{\|\mathbf{r}_l\|_2\}$ and $\{\|\mathbf{c}_l\|_{\bSigma}\}$ monotonically decreases and increases, respectively.
  \end{enumerate}
\end{theorem}

For $l\leq l_t$, detnote $\mathbf{P}_{l}=(\mathbf{p}_1,\dots,\mathbf{p}_{l+1})\in \mathbb{R}^{n_0J\times (l+1)}$ and $\mathbf{Q}_{l}=(\mathbf{q}_1,\dots,\mathbf{q}_{l})\in \mathbb{R}^{n_0J\times l}$. From \eqref{GKB1} we have
\begin{equation}{\label{eq:GKB_matForm}}
	\begin{cases}
		\beta_1\mathbf{P}_{l+1}\mathbf{e}_{1} = P_{\calN(\bSigma)^{\perp}}\mathbf{f},  \\
		\bSigma \mathbf{Q}_l = \mathbf{P}_{l+1}\mathbf{B}_l, \\
		\mathbf{P}_{l+1} = \mathbf{Q}_l\mathbf{B}_{l}^{\top}+\alpha_{l+1}\mathbf{q}_{l+1}\mathbf{e}_{l+1}^\top ,
	\end{cases}
\end{equation}
where $\mathbf{e}_{1}$ and $\mathbf{e}_{l+1}$ are the first and $(l+1)$-th columns of $\mathbf{I}_{l+1}$, and 
\begin{equation}\label{eq:B_l}
	\mathbf{B}_{l}
	=\begin{pmatrix}
		\alpha_{1} & & & \\
		\beta_{2} &\alpha_{2} & & \\
		&\beta_{3} &\ddots & \\
		& &\ddots &\alpha_{l} \\
		& & &\beta_{l+1}
		\end{pmatrix}\in  \mathbb{R}^{(l+1)\times l}
\end{equation}
has full column rank. Using \Cref{thm:gkb} and for any $\mathbf{c}\in\mathcal{K}_l$ letting $\mathbf{c}=\mathbf{Q}_l\mathbf{y}$ with $\mathbf{y}\in\mathbb{R}^{l}$, we have
\begin{align*}
  \min_{\mathbf{c}\in\mathcal{K}_l}\|\bSigma\mathbf{c}-P_{\calN(\bSigma)^{\perp}}\mathbf{f}\|_2
  &= \min_{\mathbf{y}\in\mathbb{R}^{l}}\|\bSigma \mathbf{Q}_l\mathbf{y}-\beta_1\mathbf{P}_{l+1}\mathbf{e}_{1}\|_{2} \\
  &= \min_{\mathbf{y}\in\mathbb{R}^{l}}\|\mathbf{P}_{l+1}(\mathbf{B}_l\mathbf{y}-\beta_1\mathbf{e}_{1})\|_{2}
  = \min_{\mathbf{y}\in\mathbb{R}^{l}}\|\mathbf{B}_l\mathbf{y}-\beta_1\mathbf{e}_{1}\|_{2} .
\end{align*}
Therefore, the $l$-th CG solution equals to
\begin{equation}\label{CG_sol}
  \phi_l=\pi(\mathbf{c}_l), \quad \mathbf{c}_l=\mathbf{Q}_{l}\mathbf{y}_l, \quad \mathbf{y}_l=\argmin{\mathbf{y}\in\mathbb{R}^{l}}\|\mathbf{B}_l\mathbf{y}-\beta_1\mathbf{e}_1\|_2.
\end{equation}
In other words, we only need to solve an $ l$-dimensional least squares problem at the $l$-th iteration to get the coefficient vector.

\paragraph{Early stopping criterion.} The CG iteration yields a regularized solution by early stopping. If the noise norm $\|\epsilon\|_{2}$ is available, where $\epsilon=(\epsilon_k(x_j))\in\mathbb{R}^{n_0J}$, then the discrepancy principle (DP) \cite{engl1996regularization} can be used to halt iteration at the earliest instance of $l$ that satisfies 
\begin{equation}\label{DP0}
	\|T\phi_l-\mathbf{f}\|_2=\|\bSigma \bc_{l}-\mathbf{f}\|_{2} \leq \tau\|\epsilon\|_{2},
\end{equation}
where $\tau$ is chosen to be marginally greater than 1. When $\|\epsilon\|_{2}$ is unavailable, we adopt the L-curve criterion \cite{hansen2010discrete}, which estimates the ideal early stopping iteration at the corner of the curve represented by 
\begin{equation}\label{LC0}
  \left(\log\|T\phi_l-\mathbf{f}\|_{2}, \log\|\phi_l\|_{H_{\Gbar}}\right) =
	\left(\log\|\bSigma\mathbf{c}_{l}-\mathbf{f}\|_{2}, \log\|\mathbf{c}_l\|_{\bSigma}\right) .
\end{equation}
Note from \Cref{thm:gkb} that the residual norm decreases monotonically while the solution norm increases monotonically, which together make the ``L''-shape of \eqref{LC0} possible. For the L-curve method, one must proceed a few iterations beyond the optimal $l$ to find its corner.

\paragraph{Hybrid regularization method.}
For the iterative method, the regularized solution is sensitive to the iteration number, and the DP or L-curve criterion may yield a suboptimal iteration number, resulting in an over- or under-regularized solution. To stabilize the convergence, we follow the idea of the hybrid regularization method; see e.g. \cite{Kilmer2001,Chungnagy2008}. At each iteration, instead of solving \eqref{ls_krylov}, we add an $H_{\Gbar}$-norm Tikhonov regularization term and solve the problem
\begin{equation}
  \phi_{\lambda_l} = \argmin{\phi\in\mathcal{H}_l}\|T\phi-P_{\calN(\bSigma)^{\perp}}\mathbf{f}\|_2 +
  \lambda_l \|\phi\|_{H_{\Gbar}}^2 ,
\end{equation}
where $\lambda_l$ is the regularization parameter that is updated at each iteration. For any $\phi\in\mathcal{H}_l$, using $\mathbf{B}_l$  and $\mathbf{Q}_l$ in \eqref{eq:GKB_matForm}--\eqref{eq:B_l} and letting $\phi=\pi(\bc)=\pi(\mathbf{Q}_{l}\mathbf{y})$ with  $\mathbf{y}\in\mathbb{R}^{l}$, we obtain
\begin{align*}
	\min_{\phi\in\mathcal{H}_l}\{\|T\phi-P_{\calN(\bSigma)^{\perp}}\mathbf{f}\|_{2}^{2}+\lambda_l\|\phi\|_{H_{\Gbar}}^2 \}
	&= \min_{\bc\in\mathcal{K}_l}\{\|\bSigma\bc-P_{\calN(\bSigma)^{\perp}}\mathbf{f}\|_{2}^{2}+\lambda_l\|\bc\|_{\bSigma}^2 \} \\
  &= \min_{\mathbf{y} \in \mathbb{R}^{l}}\{\|\mathbf{B}_l \mathbf{y} -\beta_1\mathbf{e}_1\|_{2}^2+\lambda_l\| \mathbf{y} \|_{2}^2\},
\end{align*}
where we have used $\mathbf{Q}_{l}^\top\bSigma \mathbf{Q}_{l}=\mathbf{I}_l$. Then, the $l$-th hybrid solution is
\begin{equation}\label{hyb_sol}
  \phi_{\lambda_l}=\pi(\mathbf{Q}_{l}\mathbf{y}_{\lambda_l}), \quad
  \mathbf{y}_{\lambda_l} = \argmin{y\in\mathbb{R}^{l}} \{\|\mathbf{B}_l \mathbf{y} -\beta_1\mathbf{e}_1\|_{2}^2+\lambda_l\| \mathbf{y} \|_{2}^2\}.
\end{equation}
Therefore, at each step we only need to update $\lambda_l$ and compute $\mathbf{y}_{\lambda_l}$, which is computationally efficient. We update $\lambda_l$ by the weighted GCV (WGCV) method; see \cite{Chungnagy2008,li2024preconditioned} for more details.

\section{Approximation from discrete data in practice}\label{sec:discreteFn}
In practice, the data are discrete, as in \eqref{eq:data}. Thus, to apply the automatic reproducing kernel, we need to numerically approximate the integrals in the exploration measure $\rho$, the automatic reproducing kernel and basis functions, and the matrix $\bSigma$ in Section \ref{sec.2}. 

The starting point is to approximate the function $\{ g[u_k](x_j,\cdot)\}_{k,j=1}^{n_0,J}$. Note that the explicit forms of these functions are unavailable since the data only provides $\{u_k(y_i)\}_{i=1}^{3J}$, values of these functions at finitely many points, but the functions $u_k$ are unknown. 
 For the operators in Examples \ref{example:IntOpt}--\ref{example:aggregation}, the data only defines $g[u_k](x_j,s_l)$ with $s_l = l/J$ for $1\leq l\leq J$, and one may use a rougher mesh for $s$ than these $J$ points. For generality, let the mesh points for $s$ be $\{s_l\}_{l=1}^{n_s}$. We denote values of the function $\{g[u_k](x_j,\cdot)\}_{k,j=1}^{n_0,J}$ at these mesh points by a vector 
\[
 \quad \bg_{kj} = (g[u_k](x_j,s_l))_{l\leq n_s} \in \R^{1\times n_s}, \quad 1\leq k\leq n_0, 1\leq j\leq J. 
\]  
The functions $\{g[u_k](x_j,\cdot)\}$ can then be approximated by various approaches, such as splines, wavelets, or Fourier series. For simplicity, we consider piece-wise constant approximations, i.e., 
\begin{equation}\label{eq:g_approx}
\widehat g_{kj}(s)= \sum_{l=1}^{n_s}  g[u_k](x_j,s_l) \mathbf{1}_{I_l}(s),	
\end{equation}
 with $I_l= (s_{l-1},s_l]$ with $s_0=0$. 
 
 Correspondingly, we use the Riemann sum to approximate the integrals of $\rho$ in \eqref{eq:exp_measure}, $\Gbar$ in \eqref{eq:G-rho}, and $\xi_{kj}$ in \eqref{eq:auto-basis}. Table \ref{tab:Data-approx} presents their approximations using semi-continuum and discrete data. The density of $\rho$ with the semi-continuum data is $\dot \rho(s)  \propto \frac{1}{n_0J}\sum_{kj} |g[u_k](x_j,s)| $, whose approximation is  
\begin{equation}
\begin{aligned}
\widehat  {\dot \rho}(s)  &  \propto \frac{1}{n_0J} \sum_{kjl} |g[u_k](x_j,s_l)| \mathbf{1}_{I_l}(s) =  \frac{1}{n_0J}  \sum_{kj}|\widehat g_{kj}(s) |,  \\
\Rightarrow \brho_{_D} & \propto \sum_{kj}| \bg_{kj}| \in \R^{1\times n_s}. 
\end{aligned}
\end{equation}
Here, the factor $\frac{1}{J}=|\Delta x|$ since the $x$-grid is uniform.  
Note that $\brho_{_D}$ is a discrete representation of the probability measure $\rho$, and it assigns weights to the sample points $\{s_l\}$. Since we use piece-wise constant approximations, these weights are the probability of $\rho$ on the sets $\{I_l\}$.

Similarly, we approximate the integral kernel $G(s,s') := \frac{1}{n_0J}\sum_{kj} g[u_k](x_j,s) g[u_k](x_j,s') $ in \eqref{eq:G-rho} by 
\begin{align*}
\widehat{ G}_{_D}(s,s') :  & = \frac{1}{n_0J} \sum_{k,j,l, l'} g_{kj}(s_l) g_{kj}(s_{l'}) \mathbf{1}_{I_l}(s) \mathbf{1}_{I_{l'}}(s) = \frac{1}{n_0J} \sum_{kj} \widehat g_{kj}(s)\widehat g_{kj}(s'), \\
 \Rightarrow   \bG_{_D} &=  \frac{1}{n_0J}\sum_{kj} \bg_{kj}^\top  \bg_{kj} =  \frac{1}{n_0J} \bg^\top \bg \in \R^{n_s\times n_s}. 
\end{align*}
Then,  $\widehat \Gbar_{_D}$ and $\overline{\bG}$ follows directly from the above approximations of $\rho$ and $G$, as in Table \ref{tab:Data-approx}.  

Lastly, each automatic basis functions $\xi_{kj}= \int_\calS \Gbar_{_D}(s,s')g[u_k](x_j,s')ds' $ in \eqref{eq:auto-basis} has approximations and discrete representations based on $\Gbar$ and $\overline{\bG}_{_D}$ as follows, 
\begin{equation}\label{eq:xi_kj_D}
\begin{aligned}
    \widehat{\xi_{kj}^{_D}}(s) & =   \frac{1}{n_0J} \sum_{k',j',l, l'} g_{k'j'}(s_l) g_{k'j'}(s_{l'})  g_{kj}(s_{l'}) |\Delta s| \mathbf{1}_{I_l}(s),  \\
	 \Rightarrow \bxi_{kj} & = \bg_{kj} \overline{\bG}_{_D} |\Delta s| \in \R^{1\times n_s}. 
\end{aligned}
\end{equation}
To approximate the normal matrix $\bSigma = ( \innerp{\xi_{kj},\xi_{k'j'}}_{H_{\Gbar}}) \in \R^{n_0J\times n_0J}$, recall that Lemma \ref{thm:FSOI}(c) implies $\innerp{\LGbar\psi,\phi}_{H_{\Gbar}} = \innerp{\psi,\phi}_{L^2_\rho}$ for any $\psi\in L^2_\rho$. Then, we obtain from \eqref{eq:auto-basis}  that  
\begin{align*}
 \innerp{ \xi_{kj}, \xi_{k'j'} }_{H_{\Gbar}}  
  & = \innerp{\LGbar^{-1}\xi_{kj}, \xi_{k'j'} }_{L^2_\rho} 
  =  \innerp{\frac{g[u_k](x_j,\cdot)}{\dot\rho(\cdot)}, \xi_{k'j'} }_{L^2_\rho}  \\
  & = \int_{\calS} g[u_k](x_j,s)\xi_{k'j'}(s)ds 	 \approx \int_{\calS} g[u_k](x_j,s)\widehat{\xi_{k'j'}^{_D}}(s)ds \approx\bg_{kj} \bxi_{k'j'}^\top \Delta s. 
\end{align*}
In other words, the discrete representation of $\bSigma$ is 
\[  \bSigma_{_D} : = \bg \bxi^\top \Delta s,\quad  \text{ where }\bg = (\bg_{kj})\in \R^{n_0J\times n_s},\, \bxi= (\bxi_{kj}) \in \R^{n_0J\times n_s}. 
\] 
 
To conclude, our estimator \eqref{eq:tik-estimator} in computational practice is  
\begin{equation}\label{eq:estimator-practice}
	\widehat \phi_\lambda =  \sum_{kj} \widehat c_{kj} \widehat{\xi}_{kj}^{_D}  = \sum_{l=1}^{n_s}  \widehat\bphi_{\lambda}(l)\mathbf{1}_{I_l}(s)  \Leftrightarrow \widehat\bphi_{\lambda} =  \widehat{\mathbf{c}}^\top_\lambda  \bxi \,\text{ with } \widehat{\mathbf{c}}_\lambda = 	( \bSigma_{_D}^2 + n_0J \lambda \bSigma_{_D})^{\dagger} \bSigma_{_D} \mathbf{f}, 
\end{equation}
where the basis functions $\{\mathbf{1}_{I_l}(s)\}_{l=1}^{n_s}$ originate from the piecewise constant approximation of the functions $\{g[u_k](x_j,\cdot)\}$ in \eqref{eq:g_approx}.

We summarize the above approximations from discrete data in Table \ref{tab:Data-approx}. 
\begin{table}[htb]
  \caption{Functions and arrays from the semi-continuum and discrete data. 
  }  \label{tab:Data-approx}
  {\renewcommand{\arraystretch}{1.5}%
  {\centering \footnotesize
  \begin{tabular}{ l |  l  | l }
  \toprule
  \hline 
  Semi-continuum Data  &  Discrete Data  & Vector/Arrays  \\
  \hline  
  $g_{kj}(s):= g[u_k](x_j,s)$  &   $\widehat{g}_{kj}(s)= \sum_{l=1}^{n_s}  g_{kjl} \mathbf{1}_{I_l}(s)$,  & $\bg_{kj} = (g_{kjl})_{1\leq l\leq n_s} \in \R^{1\times n_s}$  
\\
             &   with $g_{kjl}:= g[u_k](x_j,s_l)$    & $\bg = (\bg_{kj})  \in \R^{n_0J\times n_s} $ 
\\
   $f_k(x)$              &    & $\mathbf{f} = (f_k(x_j))\in \R^{n_0J \times 1}$
\\ \hline 
 $\dot\rho(s) \propto \sum_{kj}  |g_{kj}(s)| \Delta x$ & $\widehat{\dot \rho}_{_D}(s)\propto   \sum_{kjl} |g_{kjl}| \Delta x \mathbf{1}_{I_l}(s)$   & $\brho_{_D} \propto \sum_{kj}| \bg_{kj}| \in \R^{1\times n_s}$  
 \\
 $G(s,s') = \frac{1}{n_0J}\sum_{kj}g_{kj}(s)g_{kj}(s')$  & $\widehat G_{_D}(s,s') $ & $\bG_{_D} = \frac{1}{n_0J} \bg^\top \bg  \in \R^{n_s\times n_s} $  
 \\ 
 $\Gbar(s,s')=  \frac{G(s,s')}{\dot \rho(s) \dot \rho(s')}$ & $
\widehat{\Gbar}_{_D}(s,s')$  & $\overline{\bG}_{_D} = \frac{\bG_{_D}}{\brho_{_D}^\top \brho_{_D}} \in \R^{n_s\times n_s} $  
\\ 
  \hline 
         $\xi_{kj}(s)  = \int \Gbar(s,s')g_{kj}(s')ds'$  & $\widehat \xi_{kj}^{_D}(s)$  & $\bxi_{kj} = \bg_{kj} \overline{\bG}_{_D} |\Delta s| \in \R^{1\times n_s} $   \\
           &      $ \hspace{9mm} = \sum_{l=1}^{n_s}\bxi_{kj}(l) \mathbf{1}_{I_l}(s)$               & $\bxi= \bg \overline{\bG}_{_D} |\Delta s|\in \R^{n_0J\times n_s} $ \\
         $\bSigma =(\innerp{ \xi_{kj}, \xi_{k'j'}}_{H_{\Gbar}} ) $ &  & $\bSigma_{_D}=\bg  \bxi^\top |\Delta s|  \in \R^{n_0J\times n_0J} $ \\
  \hline
 $ \mathcal{E}_\mathcal{D}(\phi)$ &   \multicolumn{2}{|c}{$\widehat{\mathcal{E}_\mathcal{D}}(\phi) = \widehat{\mathcal{E}_\mathcal{D}}(\bc) = \frac{1}{n_0J}\|\bSigma_{_D}\bc -\bf \|^2 $ }\\  
\hline\hline 
  \multicolumn{3}{c}{ LSE mini-norm \hspace{27mm}  $\widetilde  \phi =  \sum_{kj} \widetilde c_{kj}\widehat  \xi_{kj}^{_D} \Leftrightarrow \hspace{3mm}\widetilde\bphi =  \widetilde{\mathbf{c}}^\top  \bxi$ with $ \widetilde{\mathbf{c}} = 	\bSigma_{_D}^\dag \mathbf{f}$  \hspace{27mm} }\\ 
  \multicolumn{3}{c}{  Tikhonov Est.\hspace{27mm}  $\widehat \phi_\lambda =  \sum_{kj} \widehat c_{kj,\lambda}\widehat  \xi_{kj}^{_D} \Leftrightarrow \hspace{2mm} \widehat\bphi_{\lambda} =  \widehat{\mathbf{c}}^\top_\lambda  \bxi 
  $ with $ \widehat{\mathbf{c}}_\lambda = 	( \bSigma_{_D}^2 + n_0J \lambda \bSigma_{_D})^{\dagger} \bSigma_{_D} \mathbf{f}$  }    \\
  \hline 
  \bottomrule 
  \end{tabular} 
  }
  }
\end{table}

Note that $\bSigma_{_D}$ is singular when $n_s<n_0J$. In other words, when estimating $\phi$ at $n_s$ evaluation points, the number of necessary features (basis functions) is at most $n_s$, so the $n_0J$ data-deduced basis functions must be linearly dependent. Consequently, in this case, it is important to not use the ridge estimator $ \mathbf{c}_{ridge} = (\bSigma_{_D} + n_0J \lambda I )^{-1}\mathbf{f}$ but use $\widehat{\mathbf{c}} = 	( \bSigma_{_D}^2 + n_0J \lambda \bSigma_{_D})^{\dagger} \bSigma_{_D} \mathbf{f}$ instead.

Importantly, the automatic basis functions $\{\xi_{kj}\}$ have two major advantages over the piece-wise constants $\{ \mathbf{1}_{I_l}(s)\}_{l=1}^{n_s}$ and other spline basis functions. First, if the analytical form of the functions $\{g[u_k](x_j,\cdot)\}$ is given, we can use the automatic basis functions directly with the coefficient $\widehat{\mathbf{c}}_\lambda$ in \eqref{eq:estimator-practice}. Second, they overcome the difficulty in computing the RKHS norm. For example, if we write $\phi(s) = \sum_l \bphi(l) \mathbf{1}_{I_l}(s)$, then the regularized problem becomes $\min_{\bphi}\|\bg \bphi- \bf\|_{2}^{2}+\lambda\|\bphi\|_{C_{rkhs}}^2$, and a major difficulty is to compute the Gram matrix $C_{rkhs}= \left(\innerp{\mathbf{1}_{I_l}, \mathbf{1}_{I_{l'}}}_{H_\Gbar}\right)_{1\leq l,l'\leq n_s}$. In contrast, the Gram matrix $\bSigma_{_D}$ for the automatic basis functions is directly available.

\section{Practical algorithms for computing the estimators}\label{sec4}
When $n_0 J$ is not large, e.g., up to a few thousands, one can compute the Tikhonov regularized estimator $\widehat{\mathbf{c}}_{\lambda}=(\bSigma_{_D}^2+n_0J \lambda \bSigma_{_D})^{\dag}\bSigma_{_D}\mathbf{f}$ based on matrix decomposition. When $n_0 J$ is large, the iterative methods can efficiently compute regularized solutions. 

\subsection{Tikhonov regularization for small datasets}\label{sec:Tik}
In Tikhonov regularization, we first compute the eigenvalue decomposition: $\bSigma_{_D} = \mathbf{U}\mathbf{\Lambda} \mathbf{U}^\top$ with $\mathbf{U}= (\mathbf{u}_1,\ldots,\mathbf{u}_{n_0J})$ and $\mathbf{\Lambda}= \mathrm{diag}(\{\lambda_i\})$, where  $\lambda_1\geq\cdots\geq\lambda_{n_0J}$ are the eigenvalues of $\bSigma_{_D}$ and $\{\mathbf{u}_i\}_{i\geq 1}$ are the corresponding orthonormal eigenvectors. The solution $\widehat{\mathbf{c}}_{\lambda}=(\bSigma_{_D}^2+n_0J \lambda \bSigma_{_D})^{\dagger}\bSigma_{_D}\mathbf{f}$ can be written as
$\widehat{\mathbf{c}}_{\lambda} = \sum_{\lambda_i>0} \mathbf{u}_i \frac{\mathbf{u}_i^\top  \bf }{\lambda_i +n_0J \lambda }.$ 
Note that the components $\mathbf{u}_i^\top \bf$ corresponding to $\lambda_i=0$ do not enter the estimator. 
To handle the numerical rank-deficient of $\bSigma_{_D}$ in practical computations, we set a small threshold $\mathtt{tol}>0$ (e.g., $\mathtt{tol}=10^{-14}$ for machine precision on the order of $10^{-16}$) and let $r=\#\{\lambda_i:\lambda_{i}>\mathtt{tol}\}$ be the numerical rank of $\bSigma_{_D}$. Then, we compute a regularized $\widehat{\mathbf{c}}_{\lambda}\in\mathrm{span}\{\mathbf{u}_{1},\ldots, \mathbf{u}_r\}$ as follows. Let $\mathbf{U}_r=(\mathbf{u}_1,\ldots,\mathbf{u}_{r}) $ and $\mathbf{\Lambda}= \mathrm{diag}(\lambda_1,\ldots,\lambda_r)$. With $\bc=\mathbf{U}_{r}\mathbf{y}$ and $\mathbf{\Lambda}_{r}^{1/2}\mathbf{y}=\mathbf{z}$, the regularization problem $\min_{\bc\in\mathcal{R}(\bSigma_{_D})}\{\|\bSigma_{_D}\bc-\bf\|_{2}^2+\lambda\bc^{\top}\bSigma_{_D}\bc\}$ becomes 
\begin{equation}\label{tikh2}
  \min_{\mathbf{y}\in\mathbb{R}^{r}}\{\|\mathbf{U}_{r}\mathbf{\Lambda}_{r}\mathbf{y}-\bf\|_{2}^{2}+\lambda\|\mathbf{\Lambda}_{r}^{1/2}\mathbf{y}\|_{2}^2\} \ \ \Leftrightarrow \ \ 
  \min_{\mathbf{z}\in\mathbb{R}^{r}}\{\|\mathbf{U}_{r}\mathbf{\Lambda}_{r}^{1/2}\mathbf{z}-\bf\|_{2}^{2}+\lambda\|\mathbf{z}\|_{2}^2\} ,
\end{equation}
and $\widehat{\mathbf{c}}_{\lambda}=\mathbf{U}_{r}\mathbf{y}_{\lambda}=\mathbf{U}_{r}\mathbf{\Lambda}_{r}^{-1/2}\mathbf{z}_{\lambda}$. Finally, we obtain $\widehat{\phi}_{\lambda}=\pi(\widehat{\mathbf{c}}_{\lambda})=\sum_{kj}\widehat{c}_{kj}\xi_{kj}$. 

In order to select the optimal $\lambda$, we can use the L-curve \cite{engl1994using} or GCV criterion \cite{golub1979generalized}. The L-curve criterion plots the following parametrized curve in log-log scale:
\begin{align}
  \begin{split}
  l(\lambda) = (x(\lambda), y(\lambda)): 
 &= \left(\log(\|T\widehat{\phi}_{\lambda}-\mathbf{f}\|_{2}), \log(\|\phi_{\lambda}\|_{H_\Gbar})\right)  \\
  &= \left(\log(\|\bSigma_{_D}\widehat{\bc}_{\lambda}-\mathbf{f}\|_{2}), \log((\widehat{\bc}_{\lambda}^{\top}\bSigma_{_D}\widehat{\bc}_{\lambda})^{\frac{1}{2}})\right) ,
  \end{split}
\end{align}
and the corner of $l(\lambda)$ corresponds to a good estimate. In practical computation, we restrict $\lambda$ in the spectral range of $\bSigma_{_D}$, and compute 
\begin{equation}\label{estim_lamb}
  \lambda^{*} = \argmax{\lambda_{r}\leq\lambda\leq\lambda_{1}} \kappa(\lambda):=
  \frac{x'y''-y'x''}{(x'^{2}+y'^{2})^{3/2}}
\end{equation}
as the optimal $\lambda$ by maximizing the signed curvature of the L-curve. For the GCV criterion, by noting that 
 \[ 
\mathbf{f}-\bSigma_{_D}\widehat{\mathbf{c}}_{\lambda}=
 (\mathbf{I}_{n_0J}-\bSigma_{_D}\bSigma_{_{D,\lambda}})\mathbf{f},\]
where $\bSigma_{_{D,\lambda}}:=(\bSigma_{_D}^2+n_0J \lambda \bSigma_{_D})^{\dag}\bSigma_{_D}$, we have the following GCV function:
\begin{equation}
  \mathrm{GCV}(\lambda)=\frac{\|(\mathbf{I}_{n_0J}-\bSigma_{_D}\bSigma_{_{D,\lambda}})\mathbf{f}\|_{2}^{2}}{(\mathrm{trace}(\mathbf{I}_{n_0J}-\bSigma_{_D}\bSigma_{_{D,\lambda}}))^2}
  = \frac{\left(\sum_{i=1}^{r}\left(\frac{n_0J\lambda \mathbf{u}_{i}^{\top}\mathbf{f}}{\lambda_{i}^{2}+n_0J\lambda}\right)^2+\sum_{i=r+1}^{n_0J}(\mathbf{u}_{i}^{\top}\mathbf{f})^2\right)}{\left(n_0J-r+\sum_{i=1}^{r}\frac{n_0J\lambda}{\lambda_{i}^{2}+n_0J\lambda}\right)^2}
\end{equation}
where we have used the numerical rank $r$ to replace $\mathrm{rank}(\bSigma_{_D})$. The optimal $\lambda$ is estimated as the minimizer of $\mathrm{GCV}(\lambda)$.

\begin{algorithm}[htb]
	\caption{Tikhonov regularization}\label{alg:tikh}
  \begin{algorithmic}[1]
  \Require Data $\mathcal{D}=\{(u_k(x_j),f_k(x_j)), j=1,\ldots, J\}_{k=1}^{n_0} $
  \State Compute basis functions $\{\xi_{kj}\}$, assemble matrix $\bSigma_{_D}$ and vector $\mathbf{f}$
  \State Compute the eigenvalue decomposition: $\bSigma_{_D} = \mathbf{U}\mathbf{\Lambda}\mathbf{U}^\top$
  \State Estimate the optimal $\lambda$ by L-curve or GCV criterion
  \State Solve \eqref{tikh2} to get $\widehat{\mathbf{c}}_{\lambda}=(\widehat{c}_{kj})$; compute $\widehat{\phi}_{\lambda}=\pi(\widehat{\mathbf{c}}_{\lambda})=\sum_{k,j}\widehat{c}_{kj}\xi_{kj}$
 \Ensure Regularized estimator $\widehat{\phi}_{\lambda}$
 \end{algorithmic}
\end{algorithm}

The algorithm of Tikhonov regularization is summarized in \Cref{alg:tikh}. This method needs to store the matrix $\bSigma_{_D}$, with the memory usage of $O(n_0^2J^2)$. The main computational cost is the eigen-decomposition of $\bSigma_{_D}$, which has the order $O(n_0^3J^3)$.

\subsection{Iterative regularization for large datasets}\label{sec:IR}
For large datasets, iterative regularization methods that rely solely on matrix-vector products are more efficient. The algorithm is based on the GKB iteration introduced in \Cref{sec2.4}. 

In the GKB method, the bi-diagonal structure of $\mathbf{B}_l$ in \eqref{eq:B_l} allows us to update $\mathbf{c}_l$ step by step without explicitly solving $\min_{\mathbf{y}}\|\mathbf{B}_l\mathbf{y}-\beta_1\mathbf{e}_1\|$. The updating procedure is based on using Givens QR factorization to $\mathbf{B}_l$, which is very similar to the LSQR algorithm; see \cite{paige1982lsqr} for the details. In practice, we first compute an approximation $\bSigma_{_D}$ to replace $\bSigma$ in the computation, and then apply the GKB procedure to update the orthonormal basis of the solution subspace and the coefficient vector. The iteration will be stopped if the early stopping criterion is satisfied. The algorithm is summarized in \Cref{alg:iter_regu}.

The hybrid regularization algorithm proceeds in the same way, except that at each iteration we update $\lambda_l$ and recompute the regularized solution $\mathbf{y}_{\lambda_l}$ from \eqref{hyb_sol}. Accordingly, we omit its pseudo-code here.

\begin{algorithm}[htb]
	\caption{Iterative regularization by GKB}\label{alg:iter_regu}
	\begin{algorithmic}[1]
    \Require Data $\mathcal{D}=\{(u_k(x_j),f_k(x_j)), j=1,\ldots, J\}_{k=1}^{n_0} $
    \State Compute basis functions $\{\xi_{kj}\}$, assemble matrix $\bSigma_{_D}$ and vector $\mathbf{f}$
    \State \textbf{(Initialization)}
		\State Compute $\bar{\mathbf{f}}=P_{\calN(\bSigma_{_D})^{\perp}}\mathbf{f}$, $\beta_{1}=\|\bar{\mathbf{f}}\|_2$, $\mathbf{p}_1=\bar{\mathbf{f}}/\beta_{1}$
    \State Compute $\alpha_1=\|\mathbf{p}_1\|_{\bSigma_{_D}}$, $\mathbf{q}_1=\mathbf{p}_1/\alpha_1$
    \State Set $\mathbf{c}_0=\mathbf{0}$, $\mathbf{w}_1=\mathbf{q}_1$, $\bar{\varphi}_{1}=\beta_1$, $\bar{\rho}_1=\alpha_1$
		\For {$i=1,2,\dots, l_{max}$}
    \State \textbf{(GKB iteration)}
		\State $\mathbf{r}=\bSigma_{_D}\mathbf{q}_{i}-\alpha_i\mathbf{p}_{i}$, $\beta_{i+1}=\|\mathbf{r}\|_2$, $\mathbf{p}_{i+1}=\mathbf{r}/\beta_{i+1}$
		\State $\mathbf{s}=\mathbf{p}_{i+1}-\beta_{i+1}\mathbf{q}_{i}$, $\alpha_{i+1}=\|\mathbf{s}\|_{\bSigma_{_D}}$, $\mathbf{q}_{i+1}=\mathbf{s}/\alpha_{i+1}$
    \State \textbf{(Apply Givens QR factorization to $\mathbf{B}_i$)}
		\State $\rho_{i}=(\bar{\rho}_{i}^{2}+\beta_{i+1}^{2})^{1/2}$
		\State $\bar{c}_{i}=\bar{\rho}_{i}/\rho_{i}$, $\bar{s}_{i}=\beta_{i+1}/\rho_{i}$
		\State $\theta_{i+1}=\bar{s}_{i}\alpha_{i+1}$, $\bar{\rho}_{i+1}=-\bar{c}_{i}\alpha_{i+1}$
		\State $\varphi_{i}=\bar{c}_{i}\bar{\varphi}_{i}$, $\bar{\varphi}_{i+1}=\bar{s}_{i}\bar{\varphi}_{i}$
    \State \textbf{(Update the coefficient vector)}
		\State $\mathbf{c}_{i}=\mathbf{c}_{i-1}+(\varphi_{i}/\rho_{i})\mathbf{w}_{i}$, $\mathbf{w}_{i+1}=\mathbf{q}_{i+1}-(\theta_{i+1}/\rho_{i})\mathbf{w}_{i}$
		\EndFor
    \If {\textit{Early stopping criterion} is satisfied}   
     \State Terminate at the estimated iteration $l_*$, let $\hat{\bc} =\bc_{l_*}=(\widehat{c}_{kj})$
     \State Compute $\hat{\phi}=\sum \hat{c}_{kj}\xi_{kj}$
    \EndIf
    \Ensure Regularized estimator $\hat{\phi}$
	\end{algorithmic}
\end{algorithm}

At the initial iteration of both methods, we compute $P_{\calN(\bSigma_{_D})^{\perp}}\mathbf{f}$. If $\bSigma_{_D}$ has full-rank or $\mathbf{f}\in\mathcal{R}(\bSigma_{_D})$, then $P_{\calN(\bSigma_{_D})^{\perp}}\mathbf{f}=\mathbf{f}$. Otherwise, noting that $P_{\calN(\bSigma_{_D})^{\perp}}\mathbf{f}=\bSigma_{_D}^{\dag}\bSigma_{_D}\mathbf{f}$, we approximate this projection by iteratively solving the minimal 2-norm least squares problem $\min_{\mathbf{v}\in\mathbb{R}^{n_0J}}\|\bSigma_{_D} \mathbf{v} -\bSigma_{_D} \mathbf{f}\|_2$. This approximation does not require high accuracy, as the presence of noise limits the achievable final precision of the regularized estimator. In practice, it is carried out efficiently via the LSQR algorithm \cite{paige1982lsqr}.

The iterative method requires $O(n_0^2J^2)$ storage, matching the storage requirements of the direct method of Tikhonov regularization. Each iteration is dominated by the matrix-vector product with large $\bSigma_D$, costing $O(n_0^2J^2)$ operations. Thus, over $l_{\text{max}}$ iterations, the total computational complexity is $O\bigl(n_0^2J^2\,l_{\max}\bigr)$. The hybrid method also incurs a total cost of $O\bigl(n_0^2J^2\,l_{\max}\bigr)$, as the additional cost of $O(l^3)$ from the SVD of $\mathbf{B}_l$ in WGCV at each iteration is negligible compared to the dominant $O(n_0^2 J^2)$ term, placing its complexity between that of the iterative method and the direct method.

\section{Numerical experiments}\label{sec5}
We present numerical results for three examples of learning kernels in operators, including integral operators, nonlocal operators, and aggregation operators in mean-field equations, as detailed in Examples \ref{example:IntOpt}--\ref{example:aggregation}. All experiments were conducted in MATLAB R2023b using double precision. The codes are available at \url{https://github.com/Machealb/Automate-kernel}.

\paragraph*{Numerical settings.}
The input data $\{u_k\}$ are described in Examples \ref{example:IntOpt}--\ref{example:aggregation} with $n_0=30$ and $\sigma_n= n^{-2}$. They lead to ill-conditioned and rank-deficient regression matrices with eigenvalues decaying near polynomially. We use a uniform mesh with mesh size $\Delta x=0.005$. We use the Gaussian quadrature integrator for the integral in the operators to generate data, and use the Riemann sum to approximate it when computing the estimators. Unless otherwise specified, for all the examples we set the noise-to-signal ratio ($nsr$) to be $nsr=0.1$, which corresponds to a noise with standard deviation of around $\sigma =0.01$. Here the signal strength is the average $L^2_\nu$-norm of the output $\{R_\phi[u_k](x_j)\}_{k,j}$.    

The true kernels $\phi$ for the three examples are 
\begin{equation*}
\phi_1(s) = \sin(2\pi s), \; \phi_2(s) = \sin(4\pi s)\mathbf{1}_{[0,0.8]}(s), \; \phi_3(s) = -2\sin^3(6\pi s), 
\end{equation*}
respectively, and they are plotted in \Cref{fig:est}. Note that the kernel $\phi_2$ of Example \ref{example:nonlocal} has a jump discontinuity. As observed in \cite{LLA22}, estimator accuracy improves when the smoothness of data matches that of the true kernel.  Accordingly, we generate discontinuous data for Example \ref{example:nonlocal} by multiplying each smooth $u_k$ in Example \ref{example:IntOpt} by the indicator of $[-0.5,0.8]$, i.e., $u_k(y) \mapsto u_k(y)\mathbf{1}_{[-0.5,0.8]}(y)$. Furthermore, these true kernels are close to the identifiable spaces $H=\mathcal{N}(\LGbar)^\perp$ for each example, making accurate estimation possible.

For each regularized estimator, we evaluate the relative $L_{\mu}^2$-error with respect to the true solution, where $\mu$ is the Lebesgue measure. When reporting the statistics of the estimators (such as their means and box plots), we perform 50 independent simulations for each test.

\paragraph{Other regularization norms.} We benchmark our $H_\Gbar$-norm against two baseline norms for regularization: a Gaussian kernel norm $H_K$ and the $L^2_\rho$-norm. The $H_K$-norm is the norm of the RKHS with the widely-used Gaussian kernel $K(s,s')=\exp(|s-s'|^2/(2\sigma_0^2))$, where the hyperparameter is $\sigma_0=0.1$ after fine-tuning. For both RKHS norms, we use their automatic basis functions to get an $n_0J\times n_0J$ linear system, and apply the Tikhonov and iterative regularization methods to compute the estimators. For the $L^2_\rho$-norm regularization, following \Cref{sec:discreteFn}, we compute the coefficients of $\hat{\phi}(s)=\sum_{l=1}^{n_s}\hat{c}_l\mathbf{1}_{I_l}(s)$ by solving the regularized least squares problem $ \argmin{\mathbf{c}\in\mathbb{R}^{n_s}}\frac{1}{n_0J}\|\mathbf{A}\mathbf{c}-\mathbf{f}\|_{2}^{2}+ \lambda \|\mathbf{c}\|_{\mathbf{B}}^{2}$, where $\mathbf{A}=\mathbf{g}\Delta s\in\mathbb{R}^{n_0J\times n_s}$ and $\|\mathbf{c}\|_{\mathbf{B}}^{2}=\mathbf{c}^{\top}\mathbf{B}\mathbf{c}$ with $\mathbf{B}=\mathrm{diag}(\brho_{_D})$. Using the transformation $\tilde{\mathbf{c}}=\mathbf{B}^{\frac{1}{2}}\mathbf{c}$ and $\tilde{\mathbf{A}}=\mathbf{B}^{-\frac{1}{2}}\mathbf{A}$, we only need to deal with the 2-norm regularizer $\|\tilde{\mathbf{c}}\|_2$ in the Tikhonov or iterative regularization.

\paragraph{Accuracy of the estimators.}
We first compare the accuracy of the estimators computed using the three regularization norms, each with the four regularization methods: Tikhonov regularization with $\lambda$ selected by the L-curve and GCV criteria, iterative regularization with early stopping determined by the L-curve, and hybrid regularization with $\lambda_l$ updated by WGCV. 
\begin{figure}[htb]
	\centering
	 \subfloat
	{\label{fig:1a}\includegraphics[width=1\textwidth]{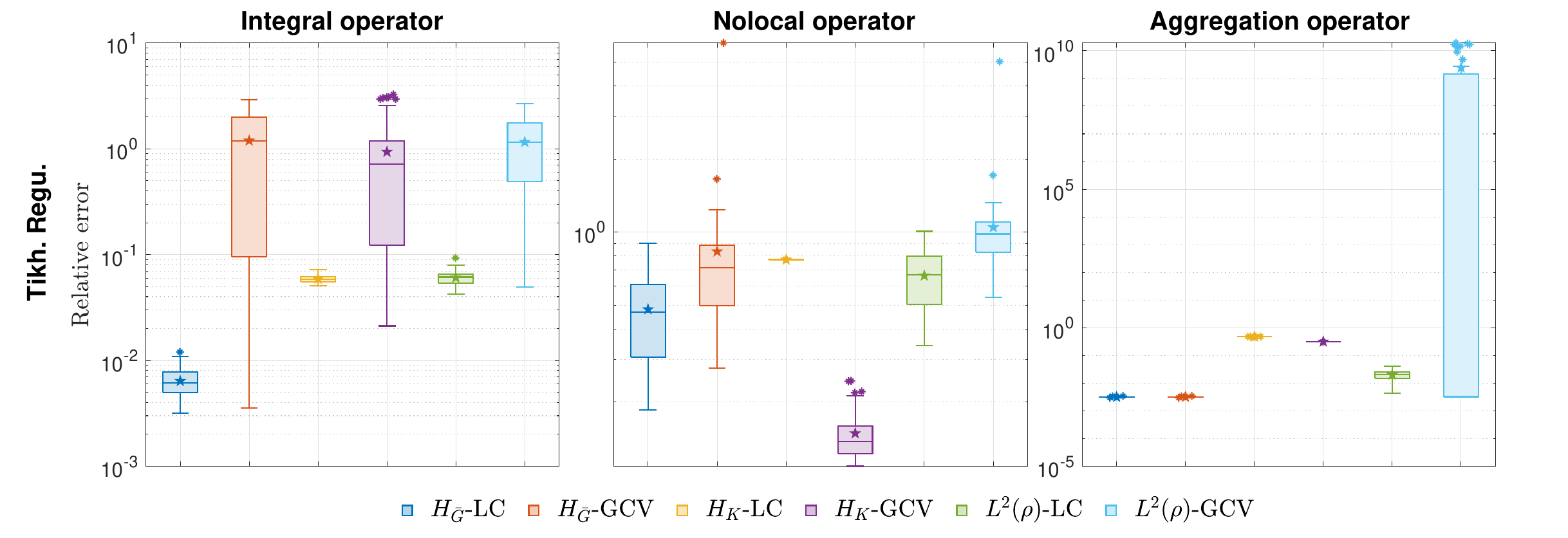}}
	\vspace{-4mm}
   \subfloat
	{\label{fig:1f}\includegraphics[width=1\textwidth]{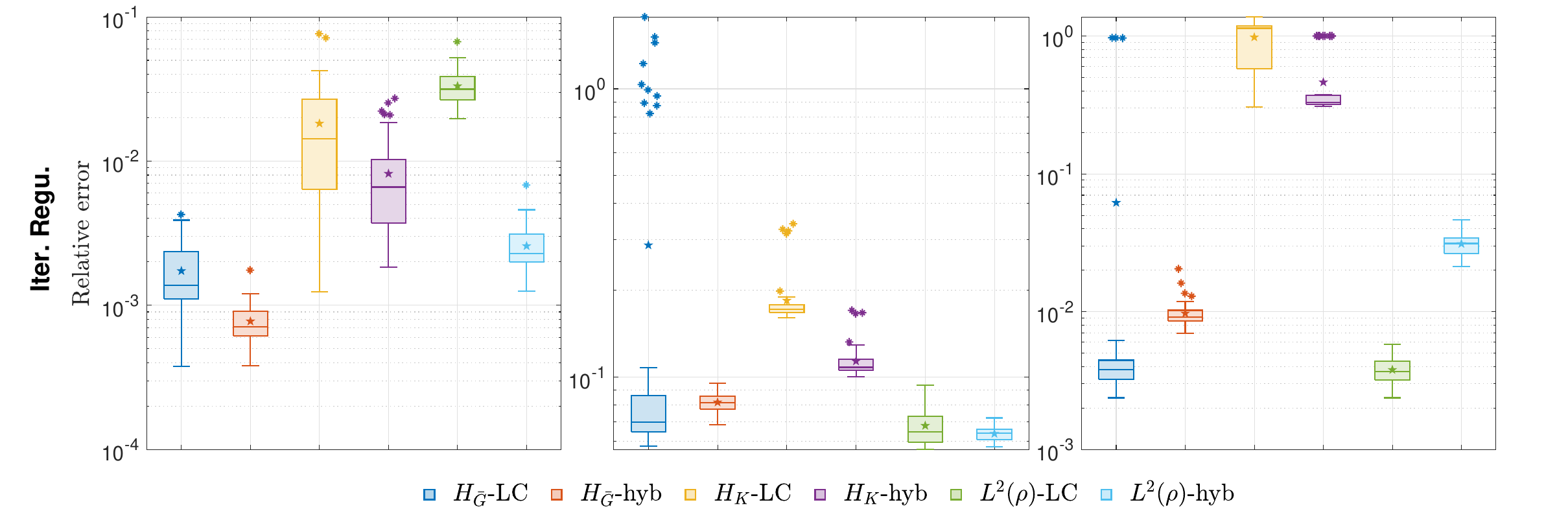}}
	\caption{Relative errors of the estimators in 50 simulations. }
	\label{fig:error}
\end{figure}

We present the $L_{\mu}^2$ ($\mu$ is the Lebesgue measure) relative errors of the estimators in 50 simulations in \Cref{fig:error}, where the box-plots display the median, lower and upper quartiles, outliers, and the minimum and maximum values that are not outliers. The abbreviation ``LC'' stands for the L-curve criterion, while ``hyb'' denotes the hybrid method.

The results demonstrate that the choice of regularization norm has a significant impact on the accuracy of the estimators. For all three examples, the $H_{\Gbar}$-norm consistently yields lower relative errors, indicating that our data-adaptive RKHS regularization can better capture the structure of the underlying nonlocal inverse problems. In contrast, the Gaussian kernel norm generally yields the largest errors with large variances. While the $L_{\rho}^2$ norm regularization occasionally achieves accuracy comparable to that of the $H_{\Gbar}$-norm, such as Example 1.2 with iterative regularization methods, it is less accurate and less stable overall.

Additionally, the L-curve and GCV methods produce comparable hyperparameter selections for Tikhonov regularization (top row of \Cref{fig:error}). However, the bottom row of \Cref{fig:error} illustrates that the purely iterative method can incur larger errors due to the instability of identifying the discrete L-curve's corner for early stopping. By contrast, the hybrid method offers greater stability and consistently achieves low errors across all cases. 

\begin{figure}[htbp]
	\centering
  \includegraphics[width=.95\textwidth]{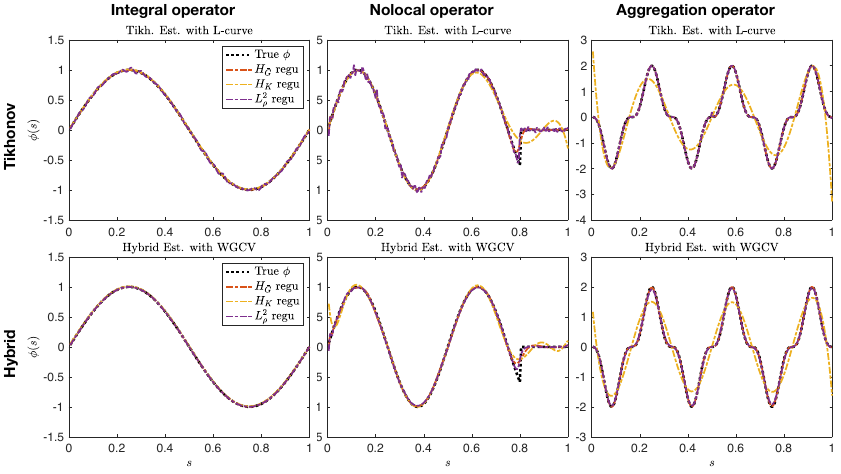}
	\caption{Typical regularized estimators by Tikhonov regularization with L-curve and hybrid regularization with WGCV using three norms: $H_\Gbar$, $H_K$, and $L^2_\rho$.}
	\label{fig:est}
\end{figure}

Figure \ref{fig:est} displays the estimators in a typical test: only the estimators of Tikhonov with L-curve and the hybrid method are shown, since the other methods yield very similar results and are omitted for clarity. Because the noise level is low, all estimators closely track the true kernels. The $H_{\Gbar}$-norm regularizer produces slightly more accurate estimates than the $L^2_\rho$ and Gaussian kernel norms. In particular, for the nonlocal operator (middle column), the $H_{\Gbar}$ regularizer better resolves the jump discontinuity than the Gaussian kernel estimator: its data-adaptive smoothness allows it to capture the discontinuity more faithfully.

\paragraph{Convergence as noise decreases.}
To compare these methods further, we examine the estimator convergence as the noise level decreases with the noise-to-signal ratio varying over $nsr\in \{1,1/2,1/4,1/8,1/16,1/32\}$ and all other settings unchanged from the previous experiment. For each noise level, we run 50 independent simulations. In Figure \ref{fig:nsrconv}, we report results for Tikhonov and iterative regularization using the L-curve for parameter selection, alongside the hybrid method; Tikhonov with GCV yields results similar to Tikhonov with L-curve and is omitted for clarity. To illustrate convergence behavior in the ideal scenario, we also include the relative error of the optimal iterative regularized solution (denoted by ``Iter.-opt.''), i.e., the solution with the minimum relative error across all iterations.

\begin{figure}[htbp]
	\centering
	\subfloat
	{\label{fig:3a}\includegraphics[width=0.95\textwidth]{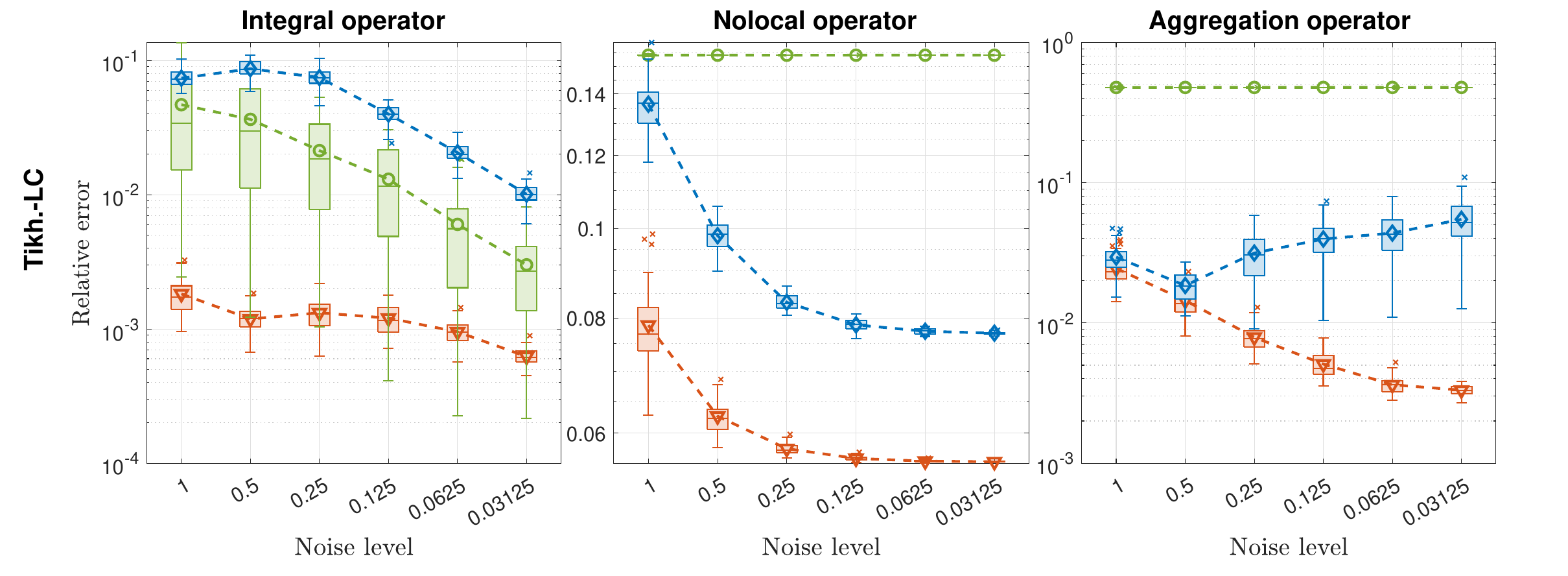}}
  \vspace{-9mm}
	\subfloat
	{\label{fig:34c}\includegraphics[width=0.95\textwidth]{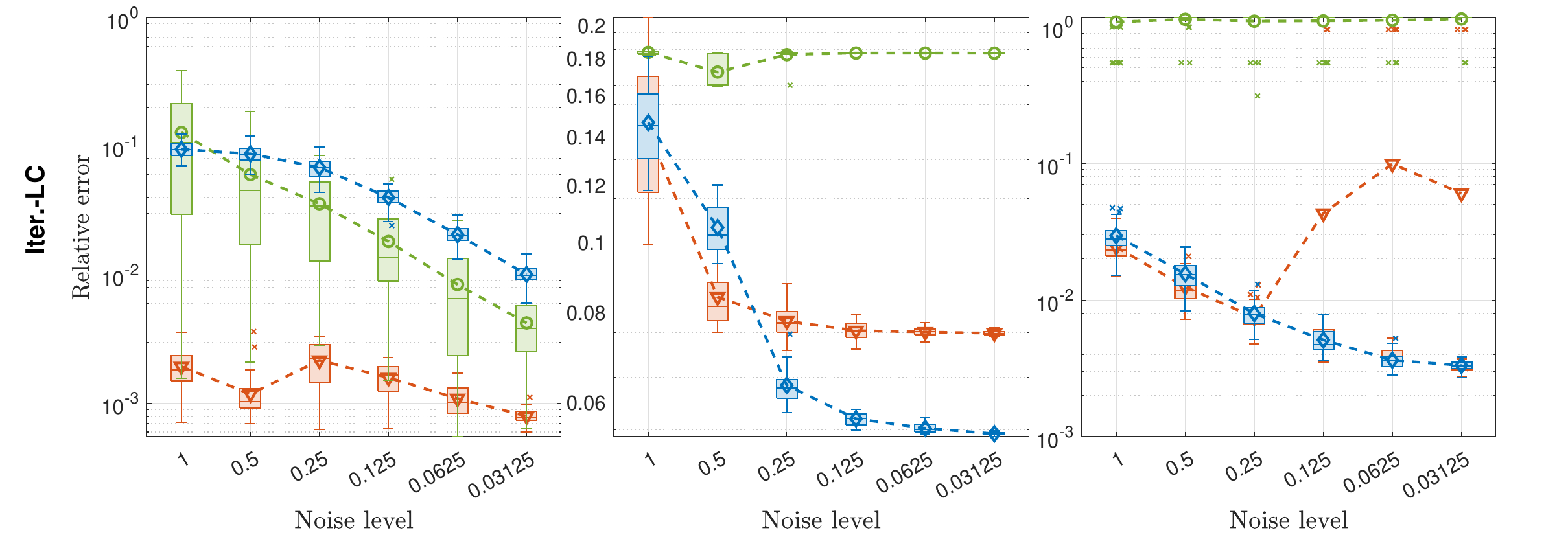}}
  \vspace{-9mm}
  \subfloat
	{\label{fig:33d}\includegraphics[width=0.95\textwidth]{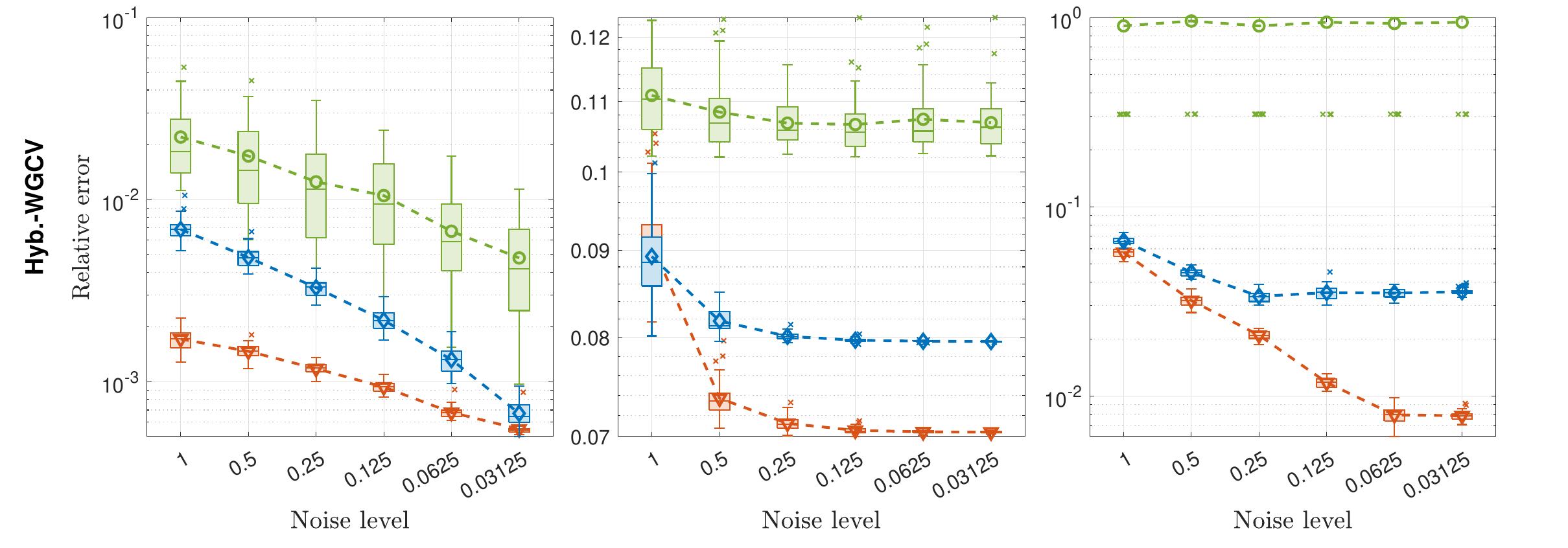}}
  \vspace{-9mm}
  \subfloat
	{\label{fig:33e}\includegraphics[width=0.95\textwidth]{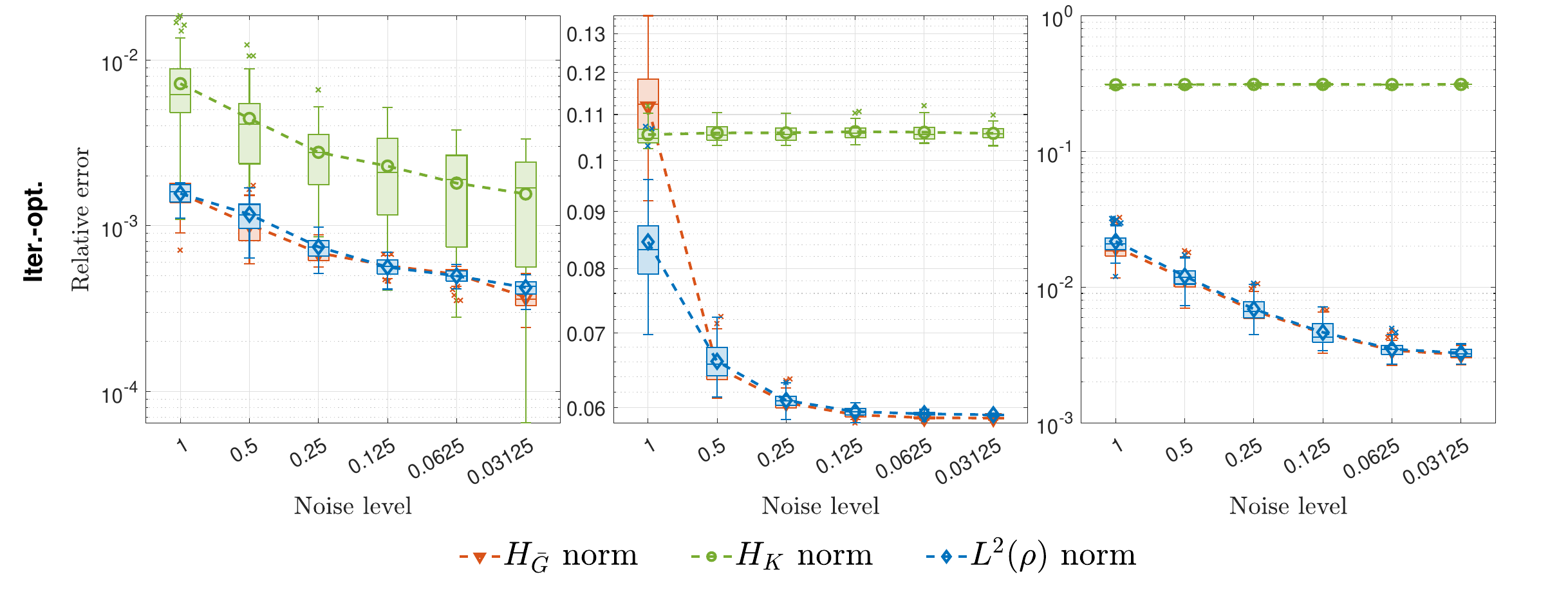}}
	\caption{Convergence of the estimators as the noise decreases in 50 simulations.}
	\label{fig:nsrconv}
\end{figure}

For the integral operator (left column of Figure \ref{fig:nsrconv}), the relative error of all estimators decreases as the noise level is reduced. For all regularization methods, the estimators obtained using the $H_{\Gbar}$-norm consistently achieve the lowest relative errors as the noise decreases. 

In particular, for the optimal iterative solutions (bottom row of Figure \ref{fig:nsrconv}), the convergence curves under the $H_{\Gbar}$ and $L_{\rho}^2$ norms are nearly identical for all the three examples, indicating that they share the same convergence rate up to a constant factor. This observation is consistent with the theoretical result presented in \cite{lu2025adaptive}. 

For the nonlocal operator (middle column of Figure \ref{fig:nsrconv}), the Gaussian kernel regularized estimators have relatively large errors that fail to decay, due to the mismatched smoothness between the true kernel and the Gaussian kernel. In contrast, for all regularization methods, the $H_{\Gbar}$- and $L^2_{\rho}$-norms consistently lead to estimator errors that decay with the noise, and their convergence curves become flat as the noise level approaches the numerical integration error.  

The aggregation-operator case (right column of Figure \ref{fig:nsrconv}) highlights the differences between methods. Both Tikhonov and hybrid methods yield convergent estimators under the $H_{\Gbar}$-norm, but not under the $L^2_{\rho}$-norm. In contrast, the iterative method performs well in the $L^2_{\rho}$-norm but suffers instability in $H_{\Gbar}$-norm due to early stopping sensitivity. All methods fail to converge under the Gaussian kernel norm, due to the relatively high frequency of the true kernel.

In summary, the $H_{\Gbar}$-norm consistently leads to convergent estimators across nearly all regularization methods and achieves the lowest relative errors as noise decreases, outperforming both the $L^2_\rho$ and the Gaussian kernel norms. These results confirm its effectiveness and robustness for learning convolution kernels. Moreover, the hybrid method exhibits the strongest convergence behavior overall, underscoring its ability to automatically select optimal regularization parameters and deliver accurate solutions.

\paragraph{Computational scalability.}
In this experiment, we evaluate the computational scalability of Tikhonov and iterative regularization for learning convolution kernels as the data size increases. We vary $n_0\in \{6,12,18,24,30,36\}$, holding all other parameters fixed as in the first experiment. We only show the results for the $H_{\Gbar}$ regularization, as it has been proven to be the most effective in prior tests. For each value of $n_0$, we set the maximum number of iterations for all the three examples as $l_{\text{max}} = {30, 30, 40, 40, 50, 50}$, which is chosen to be larger than the optimal early stopping iteration. We conduct 50 independent simulations for each setting, and record the computation times on a Debian 12 desktop with 12 Intel processors.

\begin{figure}[htbp]
	\centering
  \includegraphics[width=1.0\textwidth]{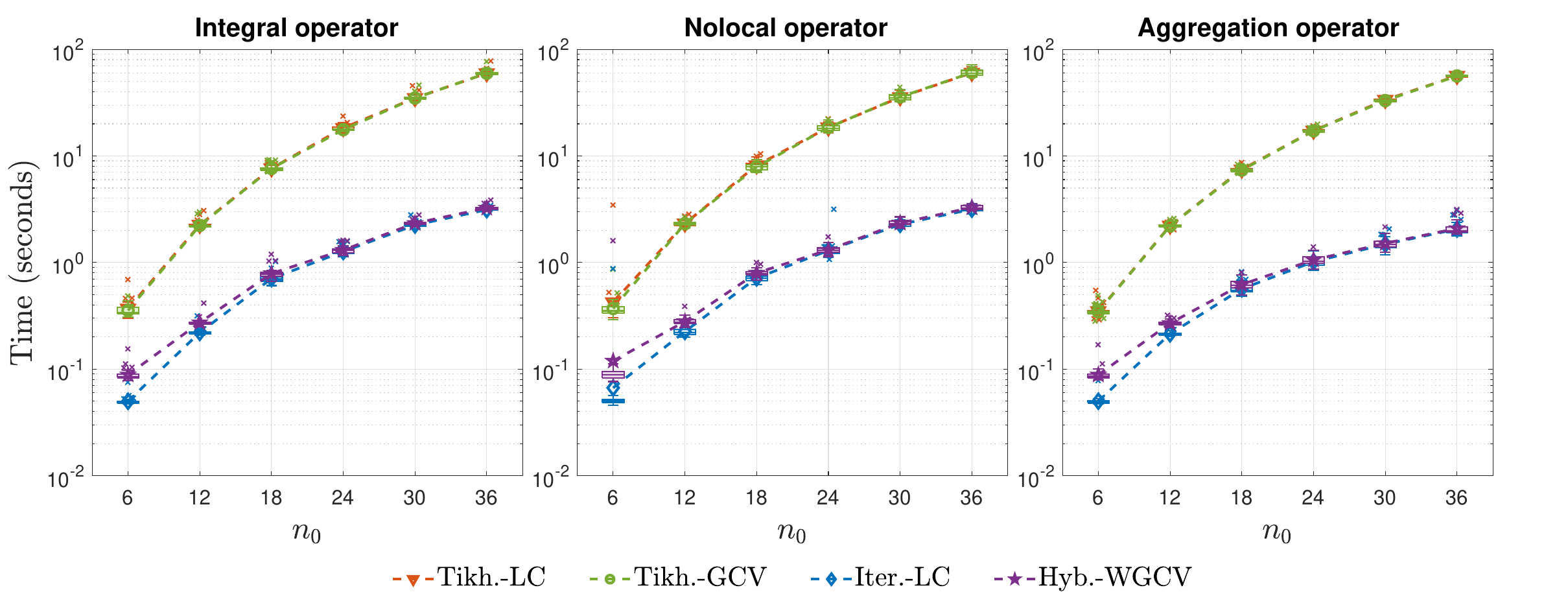}
	\vspace{-5mm}
	\caption{Running time as sample size $n_0$ increases for Tikhonov, iterative and hybrid regularization methods using $H_{\Gbar}$-norm. }
	\label{fig:njconv}
\end{figure}

\Cref{fig:njconv} reports the computation times in these tests. The iterative regularization method is orders of magnitude faster than the Tikhonov regularization method, particularly as $n_0$ increases. This behavior aligns with our theoretical analysis of the computational complexity of the two approaches. Although the hybrid method incurs slightly higher runtime than pure iterative regularization, it demonstrates significantly greater stability, as evidenced by the results of the previous experiments. Therefore, for learning convolution kernels from large datasets, the hybrid method based on iterative regularization is the most effective and reliable choice.

In summary, our numerical experiments highlight the advantages of data-adaptive RKHS regularization for learning convolution kernels. By leveraging automatically constructed basis functions, we have developed efficient and accurate iterative regularization methods that scale well with large datasets.

\section{Conclusion}\label{sec6}
We have developed robust and scalable data-adaptive (DA) RKHS regularization methods for learning convolution kernels, based on an automatic reproducing kernel that is tailored to the data and the forward operator. For discrete and finite observations, the methods use a finite set of automatic basis functions sufficient to represent minimal-norm least squares, Tikhonov, and conjugate gradient estimators in the RKHS. The DA-RKHS and automatic basis functions capture the structure imposed by the forward operator and data, enabling nonparametric and mesh-free regression without the need for reproducing kernel selection, hyperparameter tuning, or predefined bases. We have developed efficient regularization algorithms, including Tikhonov methods based on matrix decompositions for small datasets and iterative methods using only matrix-vector products for large datasets. Numerical experiments on integral, nonlocal and aggregation operators demonstrate that the proposed methods outperform the ridge regression and Gaussian process regression, highlighting their effectiveness, robustness, and scalability.

\appendix

\section{Proofs}\label{sec:appd_A}

\begin{proof}[Proof of \Cref{thm:FSOI}]
(a). It is clear that $\Gbar$ is symmetric. 
First, we show that $\Gbar$ is square-integrable.  Since $g[u_k]\in C(\calX\times \calS)$, we have 
\begin{equation} 
  \begin{aligned}
   G(s,s')&:= \int_\calX \frac{1}{n_0}\sum_{k=1}^{n_0}g[u_k](x,s)g[u_k](x,s') \, \nu(dx)\leq  C_g \dot\rho(s). 
  \end{aligned}
  \end{equation}
Then, by symmetry, we obtain that $  G(s,s') \leq C_g  \mathrm{min}\{\dot\rho(s), \dot\rho(s')\}$ for any $s,s'\in \calS$. Then, 
\[
\int_\calS \int_\calS \Gbar(s,s')^2\rho(ds)\rho(ds')  = \int_\calS \int_\calS  \frac{G(s,s')^2}{\dot\rho(s) \dot\rho(s')} dsds' 
\leq C_g^2 |\mathrm{supp}(\rho)|^2. 
\]

Then, $\LGbar$ is a compact self-adjoint operator. It is positive since 
\[
\innerp{\LGbar\phi,\phi}_{L^2_\rho} =  \int_\calS \int_\calS  \phi(s)\phi(s') G(s,s') dsds'  =\frac{1}{n_0}\sum_{k=1}^{n_0} \|R_\phi[u_k]\|^2\geq 0
\]
for any $\phi\in L^2_\rho$. 

(b). By its definition in \eqref{eq:lossFn_empirical0}, the loss function $\calE_\mathcal{D}$ can be written as \eqref{eq:lossFn-sq}. Note that 
$\innerp{\phi^D,\phi}_{L^2_\rho}  =\frac{1}{n_0}\sum_{k=1}^{n_0}\innerp{R_\phi[u_k],R_\phi[u_k] + \epsilon}_{\spaceY} $ for any $\phi\in L^2_\rho$, thus, we can write $\phi^D$ as $\phi^D=\LGbar\phi_*+ \eta$, where $\phi_*$ is the true kernel and $\eta\sim\mathcal{N}(0,\sigma_\epsilon^2\LGbar)$. In particular, when the data is noiseless, we have $\phi^D=\LGbar\phi_*$. Thus,  the loss function $\calE_{\mathcal{D}}$ has a unique minimizer $\widehat{\phi} = \LGbar^{-1}\phi^D =P_H(\phi_*)$ in $H:=\overline{\mathrm{span}\{\psi_i\}_{i:\lambda_i>0}}$. 

(c). The fact that $H_\Gbar= \LGbar^{1/2}(L^2_\rho)$ is a standard characterization of the RKHS, see, e.g.,\cite{aronszajn1950,CZ07book,LLA22}. Also, for any $\phi = \sum_{i}c_i\psi_i\in H_\Gbar$ and $\psi =\sum_{i}d_i\psi_i \in L^2_\rho$, using the fact that $\innerp{\psi_i,\psi_j}_{H_\Gbar} = \delta_{ij}\lambda_i^{-1}$, we have 
\[
\innerp{\phi,\LGbar\psi}_{H_\Gbar} = \sum_{i}\lambda_i^{-1} c_i d_i \lambda_i =  \sum_{i} c_i d_i  = \innerp{\phi,\psi}_{L^2_\rho}. 
\]
Lastly, it follows from the definition of $H$ that $H=\overline{ H_\Gbar}$. 
\end{proof}

\medskip

\begin{proof}[Proof of \Cref{thm:gkb}]
  (a) First we prove that under the canonical basis of $\mathbb{R}^{n_0J}$, it hold that $\widetilde{T}^{*}\mathbf{y}=\mathbf{y}$ for any $\mathbf{y}\in\calN(\bSigma)^{\perp}$. Since 
  \[\langle \widetilde{T}\mathbf{x}, \mathbf{y}\rangle_{2} = \langle \mathbf{x}, \widetilde{T}^{*}\mathbf{y}\rangle_{\bSigma}  \ \Leftrightarrow \ 
  \mathbf{x}^{\top}\bSigma (\widetilde{T}^{*}\mathbf{y}-\mathbf{y}) = 0, \quad \forall \ \mathbf{x}\in\calN(\bSigma)^{\perp} ,
   \]
  we have $\bSigma (\widetilde{T}^{*}y-y)=\mathbf{0}$. Using $\widetilde{T}^{*}\mathbf{y}-\mathbf{y}\in \calN(\bSigma)^{\perp}$, we obtain $\widetilde{T}^{*}\mathbf{y}-\mathbf{y}=\mathbf{0}$, which is the desired result. Using the basic property of the GKB process, $\{\mathbf{q}_{i}\}_{i=1}^m$ and $\{\mathbf{p}_{i}\}_{i=1}^m$ are the $\bSigma$-orthonormal and 2-orthonormal bases of the Krylov subspaces
  \begin{align*}
    & \mathcal{K}_{l}(\widetilde{T}^{*}\widetilde{T},\widetilde{T}^{*}P_{\calN(\bSigma)^{\perp}}\mathbf{f}) = 
    \mathrm{span}\{(\widetilde{T}^{*}\widetilde{T})^{i}\widetilde{T}^{*}P_{\calN(\bSigma)^{\perp}}\mathbf{f}\}_{i=0}^{l-1}
    = \mathrm{span}\{\bSigma^{i}P_{\calN(\bSigma)^{\perp}}\mathbf{f}\}_{i=0}^{l-1}, \\
    & \mathcal{K}_{l}(\widetilde{T}\widetilde{T}^{*},P_{\calN(\bSigma)^{\perp}}\mathbf{f}) = 
    \mathrm{span}\{(\widetilde{T}\widetilde{T}^{*})^{i}P_{\calN(\bSigma)^{\perp}}\mathbf{f}\}_{i=0}^{l-1}
    = \mathrm{span}\{\bSigma^{i}P_{\calN(\bSigma)^{\perp}}\mathbf{f}\}_{i=0}^{l-1},
  \end{align*}
  respectively. The last relation is obvious since $\bSigma^{i}P_{\calN(\bSigma)^{\perp}}=\bSigma^{i+1}\bSigma^{\dag}$.

  (b) Since $\{\mathbf{p}_{i}\}$ and $\{\mathbf{q}_{i}\}$ are 2-orthonormal and $\bSigma$-orthonormal bases, the maximum GKB iteration must not exceed the dimension $\calN(\bSigma)^{\perp}$, which is $\mathrm{rank}(\bSigma)$, that is, $l_t\leq \mathrm{rank}(\bSigma)$. Using \Cref{prop:representorIR} and that $\mathcal{H}_m=\pi(\mathcal{K}_l)$, $\mathcal{K}_l \subset \calN(\bSigma)^{\perp}$, and $\pi|_{\calN(\bSigma)^{\perp}}\rightarrow H_{\Gbar}$ is injective, there is a one-to-one correspondence between the CG for \eqref{ls_krylov} and the CG for $\min_{\bc\in\calN(\bSigma)^{\perp}}\|\widetilde{T}\bc-P_{\calN(\bSigma)^{\perp}}\mathbf{f}\|_2$. Therefore, the CG for $T$ and $\widetilde{T}$ terminate at the same step, the basic property of CG implies that $\widetilde{\phi}=\phi_{l_t}=\pi(\bc_{l_t})$.

  (c) Using $T\circ \pi =\bSigma$, we have $T\phi_{l}-\mathbf{f}=T\circ\pi(\bc_l)-\mathbf{f}=\bSigma\bc_{l}-\mathbf{f}$. Using \eqref{eq:Sigma-f} we get $\|\phi_{l}\|_{H_{\Gbar}}=\|\bc_{l}\|_{\bSigma}$. The basic property of CG for $T$ states that with a zero initial solution, the residual norm monotonically decreases and the solution norm increases. This is the last assertion.
\end{proof}

\paragraph{Acknowledgments.} The work of FL is partially funded by the NSF DMS2238486 and AFOSR FA9550-21-1-0317. The authors would like to thank Jinchao Feng for helpful discussions. 
{\small 
\bibliographystyle{plain}
\bibliography{ref_kernel_methods,ref_regularization25_03,ref_FeiLU2505}
} 
\end{document}